\documentclass[11pt]{amsart}       
\usepackage{txfonts}
\usepackage{amssymb}
\usepackage{eucal}
\usepackage{graphicx}
\usepackage{amsmath}
\usepackage{amscd}
\usepackage[all]{xy}           
\usepackage{amsfonts,latexsym}
\usepackage{xspace}
\usepackage{epsfig}
\usepackage{float}
\usepackage{color}
\usepackage{fancybox}
\usepackage{colordvi}
\usepackage{multicol}
\usepackage{colordvi}
\usepackage[colorlinks,final,backref=page,hyperindex,hypertex]{hyperref}
\usepackage[active]{srcltx} 

\usepackage{tikz}
\usepackage{graphicx}
\usepackage{bbm}
\newtheorem{theorem}{Theorem}[section]
\newtheorem{prop}[theorem]{Proposition}
\theoremstyle{definition}
\newtheorem{defn}[theorem]{Definition}
\newtheorem{lemma}[theorem]{Lemma}
\newtheorem{coro}[theorem]{Corollary}
\newtheorem{prop-def}{Proposition-Definition}[section]
\newtheorem{coro-def}{Corollary-Definition}[section]

\newtheorem{remark}[theorem]{Remark}

\newtheorem{exam}[theorem]{Example}

\newcommand{\nc}{\newcommand}
\nc{\tred}[1]{\textcolor{red}{#1}}
\nc{\tblue}[1]{\textcolor{blue}{#1}}
\nc{\tgreen}[1]{\textcolor{green}{#1}}
\nc{\tpurple}[1]{\textcolor{purple}{#1}}
\nc{\btred}[1]{\textcolor{red}{\bf #1}}
\nc{\btblue}[1]{\textcolor{blue}{\bf #1}}
\nc{\btgreen}[1]{\textcolor{green}{\bf #1}}
\nc{\btpurple}[1]{\textcolor{purple}{\bf #1}}
\nc{\NN}{{\mathbb N}}
\nc{\ncsha}{{\mbox{\cyr X}^{\mathrm NC}}} \nc{\ncshao}{{\mbox{\cyr
X}^{\mathrm NC}_0}}

\newcommand{\efootnote}[1]{}

\renewcommand{\textbf}[1]{}

\newcommand{\delete}[1]{}

\nc{\mlabel}[1]{\label{#1}}  
\nc{\mcite}[1]{\cite{#1}}  
\nc{\mref}[1]{\ref{#1}}  
\nc{\mbibitem}[1]{\bibitem{#1}} 

\delete{
\nc{\mlabel}[1]{\label{#1}{\hfill \hspace{1cm}{\bf{{\ }\hfill(#1)}}}}
\nc{\mcite}[1]{\cite{#1}{{\bf{{\ }(#1)}}}}  
\nc{\mref}[1]{\ref{#1}{{\bf{{\ }(#1)}}}}  
\nc{\mbibitem}[1]{\bibitem[\bf #1]{#1}} 
}

\nc{\opa}{\ast} \nc{\opb}{\odot} \nc{\op}{\bullet} \nc{\pa}{\frakL}
\nc{\arr}{\rightarrow} \nc{\lu}[1]{(#1)} \nc{\mult}{\mrm{mult}}
\nc{\diff}{\mathfrak{Diff}}
\nc{\opc}{\sharp}\nc{\opd}{\natural}
\nc{\ope}{\circ}
\nc{\dpt}{\mathrm{d}}
\nc{\hck}{H_{RT}}
\nc{\vdf}{\calf}
\nc{\ldf}{\calf_\ell}
\nc{\hlf}{H_\ell}
\nc{\onek}{\mathbf{1}_\bfk}
\nc{\diam}{alternating\xspace}
\nc{\Diam}{Alternating\xspace}
\nc{\cdiam}{canonical alternating\xspace}
\nc{\Cdiam}{Canonical alternating\xspace}
\nc{\AW}{\mathcal{A}}

\nc{\ari}{\mathrm{ar}}

\nc{\lef}{\mathrm{lef}}

\nc{\Sh}{\mathrm{ST}}

\nc{\Cr}{\mathrm{Cr}}

\nc{\st}{{Schr\"oder tree}\xspace}
\nc{\sts}{{Schr\"oder trees}\xspace}

\nc{\vertset}{\Omega} 

\nc{\assop}{\quad \begin{picture}(5,5)(0,0)
\line(-1,1){10}
\put(-2.2,-2.2){$\bullet$}
\line(0,-1){10}\line(1,1){10}
\end{picture} \quad \smallskip}

\nc{\operator}{\begin{picture}(5,5)(0,0)
\line(0,-1){6}
\put(-2.6,-1.8){$\bullet$}
\line(0,1){9}
\end{picture}}

\nc{\idx}{\begin{picture}(6,6)(-3,-3)
\put(0,0){\line(0,1){6}}
\put(0,0){\line(0,-1){6}}
\end{picture}}

\nc{\pb}{{\mathrm{pb}}}
\nc{\Lf}{{\mathrm{Lf}}}

\nc{\lft}{{left tree}\xspace}
\nc{\lfts}{{left trees}\xspace}

\nc{\fat}{{fundamental averaging tree}\xspace}

\nc{\fats}{{fundamental averaging trees}\xspace}
\nc{\avt}{\mathrm{Avt}}

\nc{\rass}{{\mathit{RAss}}}

\nc{\aass}{{\mathit{AAss}}}

\nc{\vin}{{\mathrm Vin}}    
\nc{\lin}{{\mathrm Lin}}    
\nc{\inv}{\mathrm{I}n}
\nc{\gensp}{V} 
\nc{\genbas}{\mathcal{V}} 
\nc{\bvp}{V_P}     
\nc{\gop}{{\,\omega\,}}     

\nc{\bin}[2]{ (_{\stackrel{\scs{#1}}{\scs{#2}}})}  
\nc{\binc}[2]{ \left (\!\! \begin{array}{c} \scs{#1}\\
    \scs{#2} \end{array}\!\! \right )}  
\nc{\bincc}[2]{  \left ( {\scs{#1} \atop
    \vspace{-1cm}\scs{#2}} \right )}  
\nc{\bs}{\bar{S}} \nc{\cosum}{\sqsubset} \nc{\la}{\longrightarrow}
\nc{\rar}{\rightarrow} \nc{\dar}{\downarrow} \nc{\dprod}{**}
\nc{\dap}[1]{\downarrow \rlap{$\scriptstyle{#1}$}}
\nc{\md}{\mathrm{dth}} \nc{\uap}[1]{\uparrow
\rlap{$\scriptstyle{#1}$}} \nc{\defeq}{\stackrel{\rm def}{=}}
\nc{\disp}[1]{\displaystyle{#1}} \nc{\dotcup}{\
\displaystyle{\bigcup^\bullet}\ } \nc{\gzeta}{\bar{\zeta}}
\nc{\hcm}{\ \hat{,}\ } \nc{\hts}{\hat{\otimes}}
\nc{\barot}{{\otimes}} \nc{\free}[1]{\bar{#1}}
\nc{\uni}[1]{\tilde{#1}} \nc{\hcirc}{\hat{\circ}} \nc{\lleft}{[}
\nc{\lright}{]} \nc{\lc}{\lfloor} \nc{\rc}{\rfloor}
\nc{\curlyl}{\left \{ \begin{array}{c} {} \\ {} \end{array}
    \right .  \!\!\!\!\!\!\!}
\nc{\curlyr}{ \!\!\!\!\!\!\!
    \left . \begin{array}{c} {} \\ {} \end{array}
    \right \} }
\nc{\longmid}{\left | \begin{array}{c} {} \\ {} \end{array}
    \right . \!\!\!\!\!\!\!}
\nc{\onetree}{\bullet} \nc{\ora}[1]{\stackrel{#1}{\rar}}
\nc{\ola}[1]{\stackrel{#1}{\la}}
\nc{\ot}{\otimes} \nc{\mot}{{{\boxtimes\,}}}
\nc{\otm}{\overline{\boxtimes}} \nc{\sprod}{\bullet}
\nc{\scs}[1]{\scriptstyle{#1}} \nc{\mrm}[1]{{\rm #1}}
\nc{\margin}[1]{\marginpar{\rm #1}}   
\nc{\dirlim}{\displaystyle{\lim_{\longrightarrow}}\,}
\nc{\invlim}{\displaystyle{\lim_{\longleftarrow}}\,}
\nc{\mvp}{\vspace{0.3cm}} \nc{\tk}{^{(k)}} \nc{\tp}{^\prime}
\nc{\ttp}{^{\prime\prime}} \nc{\svp}{\vspace{2cm}}
\nc{\vp}{\vspace{8cm}} \nc{\proofbegin}{\noindent{\bf Proof: }}
\nc{\proofend}{$\blacksquare$ \vspace{0.3cm}}
\nc{\modg}[1]{\!<\!\!{#1}\!\!>}
\nc{\intg}[1]{F_C(#1)} \nc{\lmodg}{\!
<\!\!} \nc{\rmodg}{\!\!>\!}
\nc{\cpi}{\widehat{\Pi}}
\nc{\sha}{{\mbox{\cyr X}}}  
\nc{\shap}{{\mbox{\cyrs X}}} 
\nc{\shpr}{\diamond}    
\nc{\shp}{\ast} \nc{\shplus}{\shpr^+}
\nc{\shprc}{\shpr_c}    
\nc{\msh}{\ast} \nc{\zprod}{m_0} \nc{\oprod}{m_1}
\nc{\vep}{\epsilon} \nc{\labs}{\mid\!} \nc{\rabs}{\!\mid}
\nc{\sqmon}[1]{\langle #1\rangle}

\nc{\mmbox}[1]{\mbox{\ #1\ }} \nc{\dep}{\mrm{dep}} \nc{\fp}{\mrm{FP}}
\nc{\rchar}{\mrm{char}} \nc{\End}{\mrm{End}} \nc{\Fil}{\mrm{Fil}}
\nc{\Mor}{Mor\xspace} \nc{\gmzvs}{gMZV\xspace}
\nc{\gmzv}{gMZV\xspace} \nc{\mzv}{MZV\xspace}
\nc{\mzvs}{MZVs\xspace} \nc{\Hom}{\mrm{Hom}} \nc{\id}{\mrm{id}}
\nc{\im}{\mrm{im}} \nc{\incl}{\mrm{incl}} \nc{\map}{\mrm{Map}}
\nc{\mchar}{\rm char} \nc{\nz}{\rm NZ} \nc{\supp}{\mathrm Supp}

\nc{\Alg}{\mathbf{Alg}} \nc{\Bax}{\mathbf{Bax}} \nc{\bff}{\mathbf f}
\nc{\bfk}{{\bf k}} \nc{\bfone}{{\bf 1}} \nc{\bfx}{\mathbf x}
\nc{\bfy}{\mathbf y}
\nc{\base}[1]{\bfone^{\otimes ({#1}+1)}} 
\nc{\Cat}{\mathbf{Cat}}

\nc{\detail}{\marginpar{\bf More detail}
    \noindent{\bf Need more detail!}
    \svp}
\nc{\Int}{\mathbf{Int}} \nc{\Mon}{\mathbf{Mon}}
\nc{\rbtm}{{shuffle }} \nc{\rbto}{{Rota-Baxter }}
\nc{\remarks}{\noindent{\bf Remarks: }} \nc{\Rings}{\mathbf{Rings}}
\nc{\Sets}{\mathbf{Sets}} \nc{\wtot}{\widetilde{\odot}}
\nc{\wast}{\widetilde{\ast}} \nc{\bodot}{\bar{\odot}}
\nc{\bast}{\bar{\ast}} \nc{\hodot}[1]{\odot^{#1}}
\nc{\hast}[1]{\ast^{#1}} \nc{\mal}{\mathcal{O}}
\nc{\tet}{\tilde{\ast}} \nc{\teot}{\tilde{\odot}}
\nc{\oex}{\overline{x}} \nc{\oey}{\overline{y}}
\nc{\oez}{\overline{z}} \nc{\oef}{\overline{f}}
\nc{\oea}{\overline{a}} \nc{\oeb}{\overline{b}}
\nc{\weast}[1]{\widetilde{\ast}^{#1}}
\nc{\weodot}[1]{\widetilde{\odot}^{#1}} \nc{\hstar}[1]{\star^{#1}}
\nc{\lae}{\langle} \nc{\rae}{\rangle}
\nc{\lf}{\lfloor}
\nc{\rf}{\rfloor}

\nc{\QQ}{{\mathbb Q}}
\nc{\RR}{{\mathbb R}} \nc{\ZZ}{{\mathbb Z}}

\nc{\cala}{{\mathcal A}} \nc{\calb}{{\mathcal B}}
\nc{\calc}{{\mathcal C}}
\nc{\cald}{{\mathcal D}} \nc{\cale}{{\mathcal E}}
\nc{\calf}{{\mathcal F}} \nc{\calg}{{\mathcal G}}
\nc{\calh}{{\mathcal H}} \nc{\cali}{{\mathcal I}}
\nc{\call}{{\mathcal L}} \nc{\calm}{{\mathcal M}}
\nc{\caln}{{\mathcal N}} \nc{\calo}{{\mathcal O}}
\nc{\calp}{{\mathcal P}} \nc{\calr}{{\mathcal R}}
\nc{\cals}{{\mathcal S}} \nc{\calt}{{\mathcal T}}
\nc{\calu}{{\mathcal U}} \nc{\calw}{{\mathcal W}} \nc{\calk}{{\mathcal K}}
\nc{\calx}{{\mathcal X}} \nc{\CA}{\mathcal{A}}

\nc{\fraka}{{\mathfrak a}} \nc{\frakA}{{\mathfrak A}}
\nc{\frakb}{{\mathfrak b}} \nc{\frakB}{{\mathfrak B}}
\nc{\frakD}{{\mathfrak D}} \nc{\frakF}{\mathfrak{F}}
\nc{\frakf}{{\mathfrak f}} \nc{\frakg}{{\mathfrak g}}
\nc{\frakH}{{\mathfrak H}} \nc{\frakL}{{\mathfrak L}}
\nc{\frakM}{{\mathfrak M}} \nc{\bfrakM}{\overline{\frakM}}
\nc{\frakm}{{\mathfrak m}} \nc{\frakP}{{\mathfrak P}}
\nc{\frakN}{{\mathfrak N}} \nc{\frakp}{{\mathfrak p}}
\nc{\frakS}{{\mathfrak S}} \nc{\frakT}{\mathfrak{T}}
\nc{\frakX}{{\mathfrak X}}
\nc{\BS}{\mathbb{S
}}

\font\cyr=wncyr10 \font\cyrs=wncyr7
\nc{\li}[1]{\textcolor{red}{Li:#1}}
\nc{\yi}[1]{\textcolor{blue}{Yi: #1}}
\nc{\xing}[1]{\textcolor{purple}{Xing:#1}}
\nc{\revise}[1]{\textcolor{red}{#1}}

\nc{\ID}{{\rm I}}\nc{\lbar}[1]{\overline{#1}}\nc{\bre}{{\rm bre}}
\nc{\sd}{\cals}\nc{\rb}{\rm RB}\nc{\A}{\rm A}\nc{\LL}{\rm L}\nc{\tx}{\tilde{X}}
\nc{\col}{\Delta_{\epsilon}}\nc{\mul}{m_{RT}}\nc{\ul}{u_{RT}}\nc{\epl}{\epsilon_{RT}}
\nc{\hl}{H_{RT}}\nc{\arro}[1]{#1}\nc{\px}{P_{\tx}}\nc{\pw}{P_{\mathfrak{w}}}\nc{\pl}{B^+}
\nc{\pp}{\pl}\nc{\ppp}[1]{B^+(#1)}\nc{\dw}{\diamond_{\mathfrak{w}}}\nc{\dl}{\diamond_{\rm \ell}}
\nc{\ncshaw}{\sha^{{\rm NC}}_{\mathfrak{w}}}\nc{\ncshal}{\sha^{{\rm NC}}_{{\rm \ell}}}
\nc{\ver}{\rm V}\nc{\ld}{l}\nc{\del}{\Delta_{{\rm \ell}}}\nc{\epsl}{\epsilon_{{\rm \ell}}}
\nc{\uul}{u_{{\rm \ell}}}\nc{\oneh}{\mathbf{1}}\nc{\onew}{\mathbf{1}}
\nc{\etree}{\mathbbm{1}} \nc{\conc}{m_{RT}} \nc{\subq}{\bfk Q_l} \nc{\fid}{\unlhd}  \nc{\sfid}{\lhd}
\nc{\leql}{\leq_{\text{l}}} \nc{\leqh}{\leq_{\text{h}}} \nc{\lhl}{\leq_{\text{h,l}}} \nc{\ghl}{\geq_{\text{h,l}}}
\nc{\RT}{\mathrm{RT}}

\nc{\hrtb}{H_{\RT}(X\sqcup\Omega)} \nc{\hrts}{H_{\RT}(X, \Omega)}\nc{\rts}{\mathcal{T}(X, \Omega)}\nc{\rfs}{\mathcal{F}(X, \Omega)} \nc{\cm}{\cdot_{\varepsilon}}

\newcommand{\tun}{\begin{picture}(5,0)(-2,-1)
\put(0,0){\circle*{2}}
\end{picture}}

\newcommand{\tdeux}{\begin{picture}(7,7)(0,-1)
\put(3,0){\circle*{2}}
\put(3,0){\line(0,1){5}}
\put(3,5){\circle*{2}}
\end{picture}}

\newcommand{\ttroisun}{\begin{picture}(15,8)(-5,-1)
\put(3,0){\circle*{2}}
\put(-0.65,0){$\vee$}
\put(6,7){\circle*{2}}
\put(0,7){\circle*{2}}
\end{picture}}
\newcommand{\ttroisdeux}{\begin{picture}(5,12)(-2,-1)
\put(0,0){\circle*{2}}
\put(0,0){\line(0,1){5}}
\put(0,5){\circle*{2}}
\put(0,5){\line(0,1){5}}
\put(0,10){\circle*{2}}
\end{picture}}

\newcommand{\tquatreun}{\begin{picture}(15,12)(-5,-1)
\put(3,0){\circle*{2}}
\put(-0.65,0){$\vee$}
\put(6,7){\circle*{2}}
\put(0,7){\circle*{2}}
\put(3,7){\circle*{2}}
\put(3,0){\line(0,1){7}}
\end{picture}}
\newcommand{\tquatredeux}{\begin{picture}(15,18)(-5,-1)
\put(3,0){\circle*{2}}
\put(-0.65,0){$\vee$}
\put(6,7){\circle*{2}}
\put(0,7){\circle*{2}}
\put(0,14){\circle*{2}}
\put(0,7){\line(0,1){7}}
\end{picture}}
\newcommand{\tquatretrois}{\begin{picture}(15,18)(-5,-1)
\put(3,0){\circle*{2}}
\put(-0.65,0){$\vee$}
\put(6,7){\circle*{2}}
\put(0,7){\circle*{2}}
\put(6,14){\circle*{2}}
\put(6,7){\line(0,1){7}}
\end{picture}}
\newcommand{\tquatrequatre}{\begin{picture}(15,18)(-5,-1)
\put(3,5){\circle*{2}}
\put(-0.65,5){$\vee$}
\put(6,12){\circle*{2}}
\put(0,12){\circle*{2}}
\put(3,0){\circle*{2}}
\put(3,0){\line(0,1){5}}
\end{picture}}
\newcommand{\tquatrecinq}{\begin{picture}(9,19)(-2,-1)
\put(0,0){\circle*{2}}
\put(0,0){\line(0,1){5}}
\put(0,5){\circle*{2}}
\put(0,5){\line(0,1){5}}
\put(0,10){\circle*{2}}
\put(0,10){\line(0,1){5}}
\put(0,15){\circle*{2}}
\end{picture}}

\newcommand{\tcinqdeux}{\begin{picture}(15,14)(-5,-1)
\put(3,0){\circle*{2}}
\put(-0.65,0){$\vee$}
\put(6,7){\circle*{2}}
\put(0,7){\circle*{2}}
\put(3,7){\circle*{2}}
\put(3,0){\line(0,1){7}}
\put(0,7){\line(0,1){7}}
\put(0,14){\circle*{2}}
\end{picture}}

\newcommand{\tdun}[1]
{\begin{picture}(10,5)(-2,-1)
\put(0,0){\circle*{2}}
\put(3,-2){\tiny #1}
\end{picture}}

\newcommand{\tddeux}[2]{\begin{picture}(12,5)(0,-1)
\put(3,0){\circle*{2}}
\put(3,0){\line(0,1){5}}
\put(3,5){\circle*{2}}
\put(6,-3){\tiny #1}
\put(6,3){\tiny #2}
\end{picture}}

\newcommand{\tdtroisun}[3]{\begin{picture}(20,12)(-5,-1)
\put(3,0){\circle*{2}}
\put(-0.65,0){$\vee$}
\put(6,7){\circle*{2}}
\put(0,7){\circle*{2}}
\put(5,-2){\tiny #1}
\put(8,5){\tiny #2}
\put(-6,5){\tiny #3}
\end{picture}}
\newcommand{\tdtroisdeux}[3]{\begin{picture}(12,12)(-2,-1)
\put(0,0){\circle*{2}}
\put(0,0){\line(0,1){5}}
\put(0,5){\circle*{2}}
\put(0,5){\line(0,1){5}}
\put(0,10){\circle*{2}}
\put(3,-2){\tiny #1}
\put(3,3){\tiny #2}
\put(3,9){\tiny #3}
\end{picture}}

\newcommand{\tdquatredeux}[4]{\begin{picture}(20,20)(-5,-1)
\put(3,0){\circle*{2}}
\put(-.65,0){$\vee$}
\put(6,7){\circle*{2}}
\put(0,7){\circle*{2}}
\put(0,14){\circle*{2}}
\put(0,7){\line(0,1){7}}
\put(5,-2){\tiny #1}
\put(9,5){\tiny #2}
\put(-6,5){\tiny #3}
\put(-6,12){\tiny #4}
\end{picture}}
\newcommand{\tdquatretrois}[4]{\begin{picture}(20,20)(-5,-1)
\put(3,0){\circle*{2}}
\put(-.65,0){$\vee$}
\put(6,7){\circle*{2}}
\put(0,7){\circle*{2}}
\put(6,14){\circle*{2}}
\put(6,7){\line(0,1){7}}
\put(5,-2){\tiny #1}
\put(8,5){\tiny #2}
\put(-6,5){\tiny #4}
\put(8,12){\tiny #3}
\end{picture}}

\begin{document}

\title[Weighted infinitesimal bialgebras]{Weighted infinitesimal bialgebras
}
%
\author{Yi Zhang}
\address{School of Mathematics and Statistics,
Nanjing University of Information Science \& Technology, Nanjing, Jiangsu 210044, P.\,R. China}
\email{zhangy2016@nuist.edu.cn}

\author{Xing Gao}
\address{School of Mathematics and Statistics, Key Laboratory of Applied Mathematics and Complex Systems, Lanzhou University, Lanzhou, Gansu 730000, P.\,R. China}
         \email{gaoxing@lzu.edu.cn}

\date{\today}
\begin{abstract}
As an algebraic meaning of the nonhomogenous associative Yang-Baxter equation,
weighted infinitesimal bialgebras play an important role in mathematics and mathematical physics.
In this paper, we introduce the concept of weighted infinitesimal Hopf modules and show that any module carries a natural structure of weighted infinitesimal unitary Hopf module over a weighted quasitriangular infinitesimal unitary bialgebra.
We decorate planar rooted forests  in a new way, and prove that the space of rooted forests, together with a coproduct  and a family of grafting operations, is the free $\Omega$-cocycle infinitesimal unitary bialgebra  of weight zero on a set.
A combinatorial description of the coproduct is given.
As applications, we obtain the initial object in the category of cocycle infinitesimal unitary bialgebras   on undecorated planar rooted forests, which is the object studied in the (noncommutative) Connes-Kreimer Hopf algebra. Finally, we derive two pre-Lie algebras from an arbitrary weighted infinitesimal bialgebra and weighted commutative infinitesimal bialgebra, respectively. The second construction  generalizes the  Gel'fand-Dorfman Theorem on Novikov algebras.
\end{abstract}

\subjclass[2010]{
16W99, 
05C05, 
16T10, 
16T30,  
17B60,  
81R10,  
}

\keywords{Rooted forest; Infinitesimal bialgebra; Operated algebra; Pre-Lie algebra}

\maketitle

\tableofcontents

\setcounter{section}{0}

\allowdisplaybreaks

\section{Introduction}

\footnotetext{1. This paper was started a few years ago, some other papers motivated by this paper have appeared~\cite{ZCGL18, ZGL19, ZGZ18, ZZL18}. These and other recent developments motivated  us to complete the paper properly. It should be pointed out that Refs. ~\cite{ZCGL18,ZGL19} are two different constructions and they can't cover the results proposed in Section~\ref{sec:infbi} of this paper. }

\subsection{Infinitesimal bialgebras}
The concept of infinitesimal bialgebras emerged from the pioneering work of Joni and Rota~\cite{JR},  in order  to give an algebraic framework for the calculus of Newton divided differences.
Namely, an infinitesimal bialgebra is a module $A$ which is
simultaneously an algebra (possibly without a unit) and a coalgebra (possibly without a counit),
such that the coproduct $\Delta$ is a derivation:
\begin{equation}
\Delta(ab)=a\cdot\Delta(b)+\Delta(a)\cdot b\,\text{ for }\, a, b\in A.
\label{eq:comp1}
\end{equation}
Aguiar~\cite{MA} showed that there is no non-zero infinitesimal bialgebra which is both unitary and counitary.
In that paper, Aguiar equipped an infinitesimal bialgebra with an antipode $S$ by the name of an infinitesimal Hopf algebra, taking into account most of its combinatorial properties ~\cite{Agu02} and having a wide of applications, such as Yang-Baxter equations, Drinfeld's doubles,  pre-Lie algebras and brace algebras.
The basic theory of infinitesimal bialgebras and infinitesimal Hopf algebras was originally established  in~\cite{MA, Agu01, Aguu02},
such as quasi-triangular infinitesimal bialgebras, corresponding associative Yang-Baxter equations and Drinfeld's doubles. In 2010, Bai~\cite{Bai10} introduced the concept of antisymmetric infinitesimal bialgebras, which has a close connection with dendriform D-bialgebras~\cite{Bai10}, Frobenius algebras~\cite{Bai10}, $\mathcal{O}$-operators~\cite{BGN122} and generalized AYBE~\cite{BGN12}.
Recently, Wang~\cite{WW14} generalized  Aguiar's result by developing the Drinfeld's double for braided infinitesimal Hopf algebras in Yetter-Drinfeld categories,
and Yau~\cite{Yau10} introduced infinitesimal Hom-bialgebras and further studied by  Liu, Makhlouf, Menini and Panaite~\cite{LMMP19}.

\subsection{Motivations for infinitesimal bialgebras going weighted}
In 2006, another version of infinitesimal bialgebras and infinitesimal Hopf algebras was defined by Loday and Ronco~\cite{LR06}
and brought new life on rooted trees by Foissy~\cite{Foi09, Foi10}.
More precisely,  an infinitesimal bialgebra (of this version) is a module $A$, both an associative unitary algebra and a coassociative counitary coalgebra, with
the following compatibility:
\begin{equation}
\Delta(ab)=a\cdot\Delta(b)+\Delta(a)\cdot b-a\ot b\,\text{ for }\, a, b\in A.
\label{eq:comp2}
\end{equation}

In the Ph.D. thesis~\cite{EF06}, Ebrahimi-Fard unified these two compatibilities---Eqs.~(\ref{eq:comp1}) and~(\ref{eq:comp2})---into a
weighted version:
\begin{align*}
\Delta(ab)=a\cdot\Delta(b)+\Delta(a)\cdot b+\lambda (a\ot b)\,\text{ for }\, a, b\in A,
\end{align*}
where $\lambda \in \bfk$ is a fixed constant.
This leads to the concept of weighted infinitesimal
bialgebras~\cite{EF06} . See Definitions~\ref{def:iub}  below.

Our  second motivation comes from the notation of  weighted associative classical Yang-Baxter equation~\cite{EF06}, which was rediscovered by Ogievetsky and T. Popov~\cite{OP10} under the name of non-homogenous associative classical Yang-Baxter equation. Parallel to the well-known fact that the solutions of a classical Yang-Baxter equation give rise to Lie bialgebras~\cite{Dri86}, Aguiar~\cite{MA} studied the associative Yang-Baxter equation~\cite{Pol02, Zhe98} (AYBE)
\begin{align*}
r_{13}r_{12}-r_{12}r_{23}+r_{23}r_{13}=0,
\end{align*}
and showed that any solution $r$ of AYBE in an algebra $A$ induces  an infinitesimal unitary bialgebra of weight zero. This result was generalized by Ebrahimi-Fard~\cite{EF06}
by the concept of  weighted associative classical Yang-Baxter equation:
\begin{align*}
r_{13}r_{12}-r_{12}r_{23}+r_{23}r_{13}=\lambda r_{13}.
\end{align*}
It should be pointed out that any solution $r$ of a weight ACYBE in an algebra $A$ endows $A$ with a weighted infinitesimal unitary bialgebra, involving a weighted principle derivation~\cite{EF06, OP10, ZGZ18}.
Thus the weighted infinitesimal unitary bialgebra can be regarded as  an algebraic meaning  of a weighted ACYBE.

\subsection{An interest in combinatorics}
The rooted forest is a significant object studied in combinatorics and algebra.
One of the most important examples is the Connes-Kreimer Hopf algebra of rooted forests, which was introduced and studied extensively
in~\cite{BF10, CK98, CF11, Fo02, Hof03, Moe01} and was used to treat a problem of Renormalisation in Quantum Field Theory~\cite{CK1, FGK, GPZ1, Kre98}. There are also many other Hopf algebraic structures on rooted forests, such as Foissy-Holtkamp~\cite{Foi02, Hol03}, Loday-Ranco~\cite{LR98} and Grossman-Larson~\cite{GL89}.
One reason for significance of these algebraic structures on top of rooted forests is that most of them possess universal properties.
For example, the Connes-Kreimer Hopf algebra of rooted forests
inherits its algebra structure from the initial object in the category of (commutative) algebras with a linear operator~\cite{Foi02,Moe01}.
Recently this universal property of rooted forests was generalized in~\cite{ZGG16} in terms of decorated planar rooted forests,
and the universal property of Loday-Ronco Hopf algebra was investigated in~\cite{ZG18} in terms of decorated planar binary trees.

The concept of an algebra with (one or more) linear operators mentioned in last paragraph was invented by Kurosh~\cite{Kur60}.
Later Guo~\cite{Guo09} called them $\Omega$-operated algebras and constructed the free objects in terms of various combinatorial objects, such as Motzkin paths, rooted forests and bracketed words. Here $\Omega$ is a nonempty set to index the linear operators.
See also~\cite{BCQ10, GG}. The well-known Connes-Kreimer Hopf algebra of rooted forests can be treated in the framework of operated algebras,
with the help of the grafting operation $B^+$.
Moreover, the decorated planar rooted forests $H_{\RT}(\Omega)$ whose vertices are decorated by a nonempty set $\Omega$,
together with a set of grafting operations $\{B^+_\omega\mid \omega\in \Omega\})$, is the initial object (or free object on the empty set) in  the category of $\Omega$-operated algebras~\cite{KP,ZGG16}.

The very first example of interest is the construction of an infinitesimal bialgebra of weight $\lambda$ on the polynomial
algebra $\bfk \langle x_1,\ldots ,x_n\rangle$. Also we construct infinitesimal bialgebras of weight zero
on a class of decorated planar rooted forests $\hrts$ and give a combinatorial description of the coproduct, as in the case of the coproduct in the Connes-Kreimer Hopf algebra by admissible cuts.
Motivated by the grafting operations $\{B^+_\omega\mid \omega\in \Omega\})$
on $\hrts$, we extend our framework to $\Omega$-operated algebras and propose the concepts
of $\Omega$-operated infinitesimal bialgebras (resp. Hopf algebras) of weight $\lambda$.
Further, involving an infinitesimal version of a Hochschild 1-cocycle condition and the unitary property, we pose the concepts of $\Omega$-cocycle infinitesimal unitary bialgebras (resp. Hopf algebras) of weight $\lambda$.
As usual that combinatorial objects possess some universal properties,
we prove that the decorated planar rooted forests $H_{RT}(X, \Omega)$, together with the grafting operations $\{B^+_\omega\mid \omega\in \Omega\})$,
is the free object in the category of $\Omega$-cocycle infinitesimal unitary bialgebras (resp. Hopf algebras) of weight zero.

\subsection{An interest in Rota-Baxter algebras}
The concept of the Rota-Baxter algebra originated from the probability study of G. Baxter~\cite{Bax60}, is in order to understand Spitzer's identity in fluctuation theory and further studied by some well-known mathematicians such as Atkinson, Cartier and Rota~\cite{Ro}.
 More accurately, for a given commutative ring $\bfk$ and $\lambda \in \bfk$, a Rota-Baxter algebra of weight $\lambda$, is a pair $(R,P)$ consisting of an associative algebra $R$ over $\bfk$ and a linear operator $P:R\rightarrow R$ which satisfies the Rota-Baxter equation
\begin{align}
P(x)P(y)=P(xP(y)+P(x)y+\lambda xy) \, \text{ for }\, x, y\in R. \label{eq:RBid}
\end{align}
Then $P$ is called a Rota-Baxter operator of weight $\lambda$.
The  Rota-Baxter algebra could be regarded as an algebraic framework of the integral analysis, parallel to the fact that a differential algebra could be considered as an algebraic abstraction of differential equations. Let $(C, \Delta, \varepsilon)$ be a coalgebra and $(A, m)$ be an associative algebra. Then $\mathrm{Hom}_\bfk (C, A)$ becomes an associative algebra under the convolution product
\begin{align*}
(f\star g):= m(f\ot g)\Delta \, \text{ for all } f, g \in \mathrm{Hom}_\bfk (C, A).\,
\end{align*}
Suppose  that $C$ is further a bialgebra and replace $A$ by $C$, Ebrahimi-Fard~\cite{EF06} equipped $\mathrm{Hom}_\bfk (C, A)$ with another associative multiplication, namely, a simple composition of linear maps, denoted by $(\mathrm{Hom}_\bfk (C, C), \circ)$.
Let $(C,m,\Delta)$ be an infinitesimal bialgebra of weight $\lambda$. Define linear maps
\begin{align*}
&\gamma_L: \mathrm{Hom}_\bfk (C, C) \rightarrow \mathrm{Hom}_\bfk (C, C),\,  f\mapsto \id_C\star f,  \text { and }\\
&\gamma_R: \mathrm{Hom}_\bfk (C, C) \rightarrow \mathrm{Hom}_\bfk (C, C),\,  f\mapsto f \star \id_C.
\end{align*}
Then $\gamma_L, \gamma_R$ are two  commutative Rota-Baxter operators on $(\mathrm{Hom}_\bfk (C, C), \circ)$, see~\cite[Proposition~3.14]{EF06} for more details.

\subsection{An interest in pre-Lie algebras}
Pre-Lie algebras first appeared in the work of Vinberg on convex homogeneous
cones~\cite{Vin63} and also appeared independently at the same time in the study of the cohomology of associative algebras~\cite{Ger63}.
It has remarkable connections with many areas in mathematics and mathematical physics, such as complex and symplectic structures
on Lie groups and Lie algebras, classical and quantum Yang-Baxter equations, vertex algebras, quantum field theory and operads (see~\cite{Bai,Man11}
and references therein). In the framework of Aguiar~\cite{Aguu02}, a pre-Lie algebra structure is constructed
from an infinitesimal bialgebra of weight zero. Here we generalize Aguiar's construction and derive a pre-Lie algebra from an arbitrary
infinitesimal bialgebra of weight $\lambda$. We introduce a concept of derivations of  weight $\lambda$, which generalizes the well-known classical derivation. Having this concept in hand, we generalize the Gel'fand-Dorfman theorem on Novikov algebra under a weighted version. As an application, we derive a new pre-Lie algebra from a weighted commutative infinitesimal bialgebra.

Due to the construction of  an infinitesimal unitary bialgebra of weight $\lambda$ arising from an associative algebra (Example~\ref{exam:bialgebras}~(\ref{exam:ass})),
there is a close relationship among the  associative algebras, pre-Lie algebras, Lie algebras and weighted infinitesimal unitary bialgebras. This situation can be summarized in the sense of following commutative diagram of categories
\begin{align*}
\xymatrix@C0.8em{
\text{\small{Associative algebras} }\ar[d]_{} \ar[r]^{} &\text{\small{Weighted infinitesimal unitary bialgebras} } \ar[d]^{} \\
\text{\small{Lie algebras}} &  \ar[l]^{}  \text{\small{Pre-Lie algebras} }}
\end{align*}

{\bf Structure of the Paper.} In Section~\ref{sec:ibw}, we first propose the concept of an infinitesimal (unitary) bialgebra of weight $\lambda$. Then we show that any associative algebra (with unit) has a natural  weighted infinitesimal unitary bialgebraic structure. We  also introduce  the concept of a weighted infinitesimal counitary bialgebra , which makes the emergence of a weighted augmented algebra.
As a dual of weighted derivation, we propose the concept of a weighted coderivation (Definition~\ref{def:coder}) and prove that the dual of an infinitesimal bialgebra of weight $\lambda$ is also an infinitesimal bialgebra of weight $\lambda$ (Theorem~\ref{thm:dual}).

In Section~\ref{sec:wei}, inspired by Hopf modules, we introduce the concept of a weighted infinitesimal (unitary) Hopf modules (Definition~\ref{def:ihmodule}), which generalizes the infinitesimal Hopf modules studied by Aguiar~\cite{Agu01}. We show that some basic examples of classical Hopf modules also admit  a weighted infinitesimal Hopf versions in the context of weighted infinitesimal bialgebras (Example~\ref{exam:exmo} and Proposition~\ref{prop:moexam}). We prove that any $A$-module carries a natural structure of weighted infinitesimal unitary Hopf module over $A$, when $A$ is a weighted quasitriangular infinitesimal unitary bialgebra (Theorem~\ref{thm:main}).

Section~\ref{sec:infbi} is mainly devoted to infinitesimal unitary bialgebraic structures on top of our decorated planar rooted forests $H_{\RT}(X,\Omega)$.
After recalling the basics in planar rooted forests, we first give a new way to decorate planar rooted forests,
which makes it possible to construct more general free objects.
By applying an infinitesimal version of a Hochschild 1-cocycle condition (Eq.~(\ref{eq:cdbp})),
we construct a new coproduct on $H_{\RT}(X,\Omega)$ to equip it with a new coalgebraic structure (Lemma~\ref{lem:rt1}).
A combinatorial description of this new coproduct is also given (Theorem~\ref{thm:comb}).
Further $H_{\RT}(X,\Omega)$ can be turned into an infinitesimal unitary bialgebra of weight zero with respect to the concatenation product and the empty tree as its unit (Theorem~\ref{thm:rt2}).
Combining weighted infinitesimal bialgebras with operated algebras, we propose the concept of weighted $\Omega$-operated infinitesimal bialgebras (Definition~\ref{defn:xcobi}~(\ref{it:def1})).
When an infinitesimal 1-cocycle condition is involved,  the concept of weighted $\Omega$-cocycle infinitesimal unitary bialgebras is also introduced~(Definition~\ref{defn:xcobi}~(\ref{it:def3})).
Thanks to these concepts, we show that $H_{\RT}(X,\Omega)$ is the free $\Omega$-cocycle infinitesimal unitary bialgebra of weight zero on a set $X$ (Theorem~\ref{thm:propm}).

In Section~\ref{sec:preLie}, we derive a pre-Lie algebraic structure from a weighted infinitesimal bialgebra (Theorem~\ref{thm:preL}).
As an application, we equip $\hrts$ with a pre-Lie algebraic structure $(\hrts, \rhd_{\RT})$ and a Lie algebraic structure $(\hrts, [_{-}, _{-}]_{\RT})$ (Theorem~\ref{thm:preope}). The combinatorial descriptions of $\rhd_{\RT}$ and $[_{-}, _{-}]_{\RT}$ are also given (Corollary~\ref{coro:preLcomb}).

\smallskip
{\bf Notation.}
Throughout this paper, let $\bfk$ be a unitary commutative ring unless the contrary is specified,
which will be the base ring of all modules, algebras, coalgebras, bialgebras, tensor products, as well as linear maps.
By an algebra we mean an associative algebra (possibly without unit)
and by an coalgebra we mean a coassociative coalgebra (possibly without counit).
We use the Sweedler notation:$$\Delta(a) = \sum_{(a)} a_{(1)} \ot a_{(2)}.$$
For an algebra $A$, $A\ot A$ is viewed as an $A$-bimodule in the standard way
\begin{equation}
a\cdot(b\otimes c):=ab\otimes c\,\text{ and }\, (b\otimes c)\cdot a:= b\otimes ca,
\label{eq:dota}
\end{equation}
where $a,b,c\in A$.

\section{Weighted infinitesimal bialgebras}\label{sec:ibw}
In this section, we first recall the concept of weighted infinitesimal bialgebras~\cite{EF06},
 which generalizes simultaneously the one introduced by Joni and Rota~\cite{JR} and the one initiated by Loday and Ronco~\cite{LR06}. We then investigate the basic properties of weighted infinitesimal bialgebras.

\subsection{Basic definitions and  examples}

\begin{defn}
Let $\lambda$ be a given element of $\bfk$.
\begin{enumerate}
\item An {\bf infinitesimal bialgebra} (abbreviated {\bf $\epsilon$-bialgebra}) {\bf of weight $\lambda$} is a triple $(A,\mu,\Delta)$
consisting of an algebra $(A,\mu)$ and a coalgebra $(A,\Delta)$ that satisfies
\begin{equation}
\Delta (ab)=a\cdot \Delta(b)+\Delta(a) \cdot b+\lambda (a\ot b),\quad \forall a, b\in A.
\label{eq:wib}
\end{equation}
\item If further $(A,\mu,1)$ is a unitary algebra, then the quadruple $(A,\mu,1, \Delta)$ is called an {\bf infinitesimal unitary bialgebra of weight $\lambda$}.
\item If further $(A, \Delta, \varepsilon)$ is a counitary coalgebra, then the quadruple $(A,\mu,\Delta, \varepsilon)$ is called an {\bf infinitesimal counitary bialgebra of weight $\lambda$}.
\end{enumerate}
\label{def:iub}
\end{defn}

{\it For the rest of this paper}, we will
use the infix notation $\epsilon$-interchangeably with the adjective ``infinitesimal".

\begin{remark}
\begin{enumerate}
\item  This involving of weight allows us to unify the two versions of comparability Eqs.~(\ref{eq:comp1})
and~(\ref{eq:comp2}) employed respectively in~~\cite{JR} and~~\cite{LR06}.
More precisely, the $\epsilon$-bialgebra introduced by Joni and Rota~\cite{JR} is of weight 0,
and the $\epsilon$-bialgebra originated from Loday and Ronco~\cite{LR06} is of weight $-1$.

\item The aim to involve the unitary property is to use the infinitesimal 1-cocycle condition (Eq.~(\ref{eq:cdbp}) below), which needs the unit to construct a coproduct in an $\epsilon$-bialgebra on decorated planar rooted forests.
    Indeed, the infinitesimal bialgebra introduced by Loday and Ronco~\cite{LR06} involves the unitary and counitary properties. Here the requirement $\lambda\neq 0$, when one incorporates the unitary and counitary properties, is due to the fact that there are no non-zero
$\epsilon$-bialgebras of weight zero which are both unitary and counitary~\cite{MA}.
\end{enumerate}
\end{remark}

\begin{defn}
Let $A$ and  $B$ be two $\epsilon$-bialgebras of weight $\lambda$.
A map $\phi : A\rightarrow B$ is called an {\bf infinitesimal bialgebra morphism} if $\phi$ is an algebra morphism and a coalgebra morphism. The concept of an {\bf infinitesimal unitary bialgebra morphism} can be defined in the same way.
\end{defn}

\begin{remark}\label{remk:units}
Let $(A,m,1, \Delta)$ be an $\epsilon$-unitary bialgebra of weight $\lambda$. Then $\Delta(1)=-\lambda(1\ot1)$, as
    \begin{align*}
    \Delta(1)=\Delta(1\cdot1)=1 \cdot \Delta(1) +\Delta(1)\cdot 1 +\lambda (1\ot 1)=2\Delta(1)+\lambda (1\ot 1).
    \end{align*}
\end{remark}

\begin{exam}\label{exam:bialgebras}
Here are some examples of weighted $\epsilon$-unitary bialgebras.
\begin{enumerate}
\item \label{exam:ass} Any unitary algebra $(A, \mu,1)$ is an $\epsilon$-unitary bialgebra of weight $\lambda$ by taking $$\Delta(a)=-\lambda(a\ot 1) \, \text{ for }\,  a\in A.$$
\item \cite[Example~2.3.5]{MA} The polynomial algebra $\bfk \langle x_1, x_2, x_3,\ldots \rangle$ is an $\epsilon$-bialgebra of weight zero
with the coproduct $\Delta$
given by Eq.~(\ref{eq:wib}) and
 \begin{align*}
    \Delta (x_n)=\sum_{i=0}^{n-1}x_{i}\ot x_{n-1-i}=1\ot x_{n-1}+x_1\ot x_{n-2}+ \cdots +x_{n-1}\ot 1,
    \end{align*}
where we set $x_0=1$. Indeed it is further an $\epsilon$-unitary bialgebra of weight zero.

\item \cite[Example~2.3.2]{MA} A quiver $Q=(Q_0,Q_1, s,t)$ is a quadruple consisting of a set $Q_0$ of vertices, a set $Q_1$ of arrows, and two maps $s,t :Q_1\rightarrow Q_0$ which associate  each arrow $a\in Q_1$ to its source $s(a)\in Q_0$ and its target $t(a)\in Q_0$.  The path algebra $\bfk Q$ can be turned into an $\epsilon$-bialgebra of weight zero with the coproduct $\Delta$ given by:
    $$\Delta(a_1\cdots a_n):= \left\{
    \begin{array}{ll}
    0 & \text{ if } n = 0,\\
   s(a)\ot t(a)& \text{ if } n=1,\\
    s(a_{1})\ot a_2 \cdot\cdots a_n+a_1\cdots a_{n-1}\ot t(a_n) \\
+\sum_{i=1}^{n-2}a_1\cdots a_i\ot a_{i+2}\cdots a_n & \text{ if } n\geq 2,
    \end{array}
    \right.$$
 where $a_1\cdots a_n$ is a path in $\bfk Q$. Here we use the convention that $a_1\cdots a_n\in Q_0$ when $n=0$.

\item \cite[Section~1.4]{Foi08} Let $(A, m, 1,\Delta, \varepsilon, c)$ be a braided bialgebra with $A = \bfk\oplus \ker\varepsilon$ and the braiding
$c:A\ot A \to A\ot A$ given by
\begin{align*}
c:\left\{
    \begin{array}{ll}
    1\ot 1 \mapsto 1\ot 1,\\
   a\ot 1\mapsto 1\ot a, \\
    1\ot b \mapsto b\ot 1,\\
    a\ot b \mapsto 0,
    \end{array}
    \right.
\end{align*}
where $a,b\in \ker \varepsilon$. Then $(A, m, 1,\Delta, \varepsilon)$ is an $\epsilon$-unitary bialgebra of weight $-1$.
\item \cite[Section~2.3]{LR06}\label{exam:tensor}
Let $V$ denote a vector space. Recall that the tensor algebra $T(V)$ over $V$ is the tensor module,
\begin{align*}
T(V)=\bfk \oplus V\oplus V^{\ot 2}\oplus \cdots \oplus V^{\ot n}\oplus \cdots,
\end{align*}
equipped with the associative multiplication called concatenation defined by
\begin{align*}
v_1\cdots v_i\ot v_{i+1}\cdots v_n \mapsto v_1\cdots v_i v_{i+1}\cdots v_n \quad \text{ for } 0 \leq i\leq n,
\end{align*}
and with the convention that $v_1v_0=1$ and $v_{n+1}v_{n}=1$. It is a well known free associative algebra. The
tensor algebra $T(V)$ is an $\epsilon$-unitary bialgebra of weight $-1$ with the coassociative coproduct defined by
\begin{align*}
\Delta(v_1\cdots v_n):=\sum_{i=0}^{n}v_1\cdots v_i\ot v_{i+1}\cdots v_n.
\end{align*}

\item A  noncommutative {\bf polynomial algebra} $\bfk \langle x_1,\ldots ,x_n\rangle$ with coefficients in $\bfk$ is a free algebra generated by $\{x_1, \ldots, x_n\}$.
Denote by
$$ \mathrm{Mon}:=\{x_{i_1}^{\alpha_1} x_{i_2}^{\alpha_2} \cdots x_{i_m}^{\alpha_m}\mid 1\leq i_1, i_2, \ldots, i_m\leq n, \alpha_k\in \mathbb{N}\}.$$
Then the elements in $\mathrm{Mon}$ are called {\bf monomials} in $\bfk \langle x_1,\ldots ,x_n\rangle$ which are the elements from the set of all words in $\{x_1, \ldots, x_n\}$. Note that $\mathrm{Mon}$ is a $\bfk$-basis of $\bfk \langle x_1\ldots ,x_n\rangle$ and $\mathrm{Mon}$ is a free monoid with the identity $x_0:=1$.
We denote the multiplication on $\bfk \langle x_1,\ldots ,x_n\rangle$ by $m$.
For any word $w\in \mathrm{Mon}$ with length $l(w)=n$, we define a new notation to choose some elements of $w$. Denote by
\begin{align*}
w[i,j]:=\begin{cases}
\text{the $i$-th element to the $j$-th element of $w$} &\text{ if }1\leq i\leq j\leq n,\\
1 &\text{ otherwise}.
\end{cases}
\end{align*}
For any word $w\in \mathrm{Mon}$, define
\begin{align*}
\col (w):=\begin{cases}
0 &\text{ if }w=0,\\
-\lambda (1\otimes 1) &\text{ if }w=1,\\
\sum_{i=1}^nw[1,i-1]\otimes w[i+1,n]
+\lambda \sum_{i=1}^{n-1}w[1,i]\otimes w[i+1,n] &\text{ if }l(w)=n>0.
\end{cases}
\end{align*}
Then $(\bfk \langle x_1,\ldots ,x_n\rangle, \col, m)$ is  an $\epsilon$-unitary bialgebra of weight $\lambda$.

\end{enumerate}
\end{exam}

\subsection{Basic properties}

We first record a lemma for a preparation.
\begin{lemma}\label{lem:cobi}
Let $(A, \mu, \Delta, \varepsilon)$ be an $\epsilon$-counitary bialgebra of weight $\lambda$. Then
\begin{align*}
\varepsilon(ab)=-\lambda\varepsilon(a)\varepsilon(b) \, \text{ for any }\, a, b \in A.
\end{align*}
\end{lemma}

\begin{proof}
For $a\in A$, we have
\begin{equation}
a = (\varepsilon \ot \id) \Delta(a) = \sum_{(a)} (\varepsilon \ot \id) \, a_{(1)}\ot a_{(2)} = \sum_{(a)} \varepsilon(a_{(1)}) a_{(2)}
\label{eq:leftcu}
\end{equation}
and so
\begin{equation}
\varepsilon(a)  = \sum_{(a)} \varepsilon(a_{(1)}) \varepsilon(a_{(2)}) = (\varepsilon \ot \varepsilon) \Delta(a).
\label{eq:leftcu1}
\end{equation}
For any $a, b \in A$,
\begin{align*}
\varepsilon(ab)=&\ (\varepsilon \ot \varepsilon)\Delta(ab) \quad\quad (\text{by Eq.~(\ref{eq:leftcu1})}) \\
=& (\varepsilon \ot \varepsilon)(a\cdot \Delta(b)+\Delta(a) \cdot b+\lambda (a\ot b))\\
=&\ (\varepsilon \ot \varepsilon)\left(\sum_{(b)}ab_{(1)}\ot b_{(2)}+\sum_{(a)}a_{(1)}\ot a_{(2)}b+\lambda (a\ot b)\right)\\
=&\sum_{(b)}\varepsilon(ab_{(1)})\varepsilon(b_{(2)})+\sum_{(a)}\varepsilon(a_{(1)}) \varepsilon(a_{(2)}b)+\lambda\varepsilon(a)\varepsilon(b)\\
=&\sum_{(b)}\varepsilon \Big( ab_{(1)}\varepsilon(b_{(2)}) \Big)+ \sum_{(a)} \varepsilon\Big(\varepsilon(a_{(1)})a_{(2)}b \Big)+\lambda\varepsilon(a)\varepsilon(b)\\
=&\ \varepsilon(ab)+\varepsilon(ab)+\lambda\varepsilon(a)\varepsilon(b)  \quad\quad (\text{by Eq.~(\ref{eq:leftcu})}) ,
\end{align*}
as required.
\end{proof}

\begin{remark}
\begin{enumerate}
\item When $\lambda=-1$, the counit $\varepsilon$ is an algebra morphism. In this case, if the algebra $A$ has a unit $1$, then $\varepsilon(1)=1_{\bfk}$.
\item The counital $\epsilon$-bialgebras introduced by Aguiar~\cite{Aguu02} is of weight zero and and the $\epsilon$-bialgebra originated from Loday and Ronco~\cite{LR06} is of weight $-1$.
\end{enumerate}
\end{remark}
This motivates the concept of an augmented algebra of weight $\lambda$, which generalizes the augmented algebras introduced by Aguiar~\cite{Aguu02}.

\begin{defn}\label{def:waa}
Let $\lambda$ be a given element of $\bfk$.
An {\bf augmented algebra} {\bf of weight $\lambda$} is a triple $(A,\mu, \varepsilon)$
consisting of an algebra $(A,\mu)$ (possibly without unit) and a linear {\bf augmentation map}
$\varepsilon: A \rightarrow \bfk$ of weight $\lambda$ satisfying
\begin{align}
\varepsilon(ab)=-\lambda\varepsilon(a)\varepsilon(b) \, \text{ for }\, a, b \in A. \label{eq:Augid}
\end{align}
\end{defn}
\begin{defn}
Let $(A, \mu_A, \varepsilon_A)$ and $(B, \mu_B, \varepsilon_B)$ be augmented algebras of weight $\lambda$.
An algebra homomorphism $f: A \rightarrow B$ is said to be {\bf augmented} if it satisfies $\varepsilon_B f=\varepsilon_A$.
\end{defn}

\begin{remark}
By Lemma~\ref{lem:cobi}, any $\epsilon$-counitary bialgebra of weight $\lambda$ is an augmented algebra of weight $\lambda$.
\end{remark}

The following result gives a way to add a counity to an $\epsilon$-algebra.

\begin{prop}
Let $(A, \mu, \Delta)$ be an $\epsilon$-bialgebra of weight $\lambda$, generated by $S$ as an algebra.
Suppose $\varepsilon: A \rightarrow \bfk$ is an augmentation map of weight $\lambda$ and
the counicity $(\varepsilon \ot \id) \Delta=\id=(  \id\ot \varepsilon) \Delta $  holds on $S$.
Then the quadruple $(A, \mu, \Delta, \varepsilon)$ is an $\epsilon$-counitary bialgebra of weight $\lambda$.
\end{prop}

\begin{proof}
It suffices to prove the counicity on $A$ by showing
\begin{align*}
(\varepsilon \ot \id) \Delta(ab)=ab\,  \text{ for } \, a,b\in S.
\end{align*}
For any $a, b\in S$, it follows from hypothesis that
$$(\varepsilon \ot \id) \Delta(a)=a\,\text{ and }\, (\varepsilon \ot \id) \Delta(b)=b,$$
whence
\begin{align*}
(\varepsilon \ot \id) \Delta(ab)=&(\varepsilon \ot \id)(a\cdot \Delta(b)+\Delta(a) \cdot b+\lambda (a\ot b))\\
=&(\varepsilon \ot \id)\left(\sum_{(b)}ab_{(1)} \ot b_{(2)}+\sum_{(a)}a_{(1)} \ot a_{(2)}b+\lambda (a\ot b)\right)\\
=&\sum_{(a)}\varepsilon(ab_{(1)}) \ot b_{(2)}+\sum_{(a)}\varepsilon(a_{(1)}) \ot a_{(2)}b+\lambda (\varepsilon(a)\ot b)\\
=&-\lambda \sum_{(a)}\varepsilon(a)\varepsilon(b_{(1)}) \ot b_{(2)}+\sum_{(a)}\varepsilon(a_{(1)}) \ot a_{(2)}b+\lambda (\varepsilon(a)\ot b)\\
=&-\lambda \varepsilon(a)b+ab+\lambda (\varepsilon(a)\ot b)\\
=&ab.
\end{align*}
Thus $\varepsilon$ is a left counit. By a similar calculation, $\varepsilon$ is a right counit.
\end{proof}

The following result shows that the tensor of two augmented algebras of weight $\lambda$ is still an augmented algebra of weight $\lambda$.
So the category of augmented algebras of weight $\lambda$ forms a monoidal category.

\begin{prop}\label{prop:aug}
Let $(A, \mu_A, \varepsilon_A)$ and $(B, \mu_B, \varepsilon_B)$ be augmented algebras of weight $\lambda$. Then $A\ot B$ is an associative algebra with the multiplication $\cm$ defined by
\begin{align}
(a_1 \ot b_1)\cm (a_2 \ot b_2):=\varepsilon_B(b_1)a_1a_2\ot b_2+\varepsilon_A(a_2)a_1\ot b_1b_2+\lambda \varepsilon_A(a_2)\varepsilon_B(b_1)a_1\ot b_2
\label{eq:Tmul}
\end{align}
for $a_i\in A, b_i\in B, 1\leq i\leq 2$.
Furthermore, $A\ot B$ is an augmented algebras of weight $\lambda$ with the augmentation map given by
\begin{align}
\varepsilon_{A\ot B}(a\ot b):=\varepsilon_A(a)\varepsilon_B(b)\, \text{ for}\, a \in A,  b\in B .
\label{eq:Taug}
\end{align}
\end{prop}
\begin{proof}
It is sufficient to check the associative law:
\begin{align*}
\Big((a_1\ot b_1)\cm (a_2\ot b_2)\Big) \cm (a_3\ot b_3)=(a_1\ot b_1)\cm \Big((a_2\ot b_2) \cm (a_3\ot b_3)\Big)
\end{align*}
for $a_i\in A, b_i\in B, 1\leq i\leq 3$.
On the one hand,
\begin{align*}
&\ \Big((a_1\ot b_1)\cm (a_2\ot b_2)\Big) \cm (a_3\ot b_3)\\
=&\ \Big(\varepsilon_B(b_1)a_1a_2\ot b_2+\varepsilon_A(a_2)a_1\ot b_1b_2+\lambda \varepsilon_A(a_2)\varepsilon_B(b_1)a_1\ot b_2\Big)\cm (a_3\ot b_3)\quad (\text{by Eq.~(\ref{eq:Tmul})})\\
=&\ \varepsilon_B(b_1)(a_1a_2\ot b_2)\cm (a_3\ot b_3)+\varepsilon_A(a_2)(a_1\ot b_1b_2) \cm (a_3\ot b_3)\\
&+\lambda \varepsilon_A(a_2)\varepsilon_B(b_1)(a_1\ot b_2) \cm (a_3\ot b_3)\\
=&\ \varepsilon_B(b_1)\Big(\varepsilon_B(b_2)a_1a_2a_3\ot b_3+\varepsilon_A(a_3)a_1a_2\ot b_2b_3+\lambda\varepsilon_A(a_3)\varepsilon_B(b_2)a_1a_2\ot b_3\Big)\\
&+\varepsilon_A(a_2)\Big(\varepsilon_B(b_1b_2)a_1a_3\ot b_3+\varepsilon_A(a_3)a_1\ot b_1b_2b_3+\lambda\varepsilon_A(a_3)\varepsilon_B(b_1b_2)a_1\ot b_3\Big)\\
&+\lambda\varepsilon_A(a_2)\varepsilon_B(b_1)\Big(\varepsilon_B(b_2)a_1a_3\ot b_3+\varepsilon_A(a_3)a_1\ot b_2b_3+\lambda\varepsilon_A(a_3)\varepsilon_B(b_2)a_1\ot b_3\Big)\\
=&\ \varepsilon_B(b_1)\varepsilon_B(b_2)a_1a_2a_3\ot b_3+\varepsilon_B(b_1)\varepsilon_A(a_3)a_1a_2\ot b_2b_3 \\ &+\lambda\varepsilon_B(b_1)\varepsilon_A(a_3)\varepsilon_B(b_2)a_1a_2\ot b_3-\lambda\varepsilon_A(a_2)\varepsilon_B(b_1)\varepsilon_B(b_2)a_1a_3\ot b_3\\
&+\varepsilon_A(a_2)\varepsilon_A(a_3)a_1\ot b_1b_2b_3 -\lambda^2\varepsilon_A(a_2)\varepsilon_A(a_3)\varepsilon_B(b_1)\varepsilon_B(b_2)a_1\ot b_3\\
&+\lambda\varepsilon_A(a_2)\varepsilon_B(b_1)\varepsilon_B(b_2)a_1a_3\ot b_3+\lambda\varepsilon_A(a_2)\varepsilon_B(b_1)\varepsilon_A(a_3)a_1\ot b_2b_3\\
&+\lambda^2\varepsilon_A(a_2)\varepsilon_B(b_1)\varepsilon_A(a_3)\varepsilon_B(b_2)a_1\ot b_3\quad (\text{by Eq.~(\ref{eq:Augid})})\\
=&\ \varepsilon_B(b_1)\varepsilon_B(b_2)a_1a_2a_3\ot b_3+\varepsilon_B(b_1)\varepsilon_A(a_3)a_1a_2\ot b_2b_3
+\lambda\varepsilon_B(b_1)\varepsilon_A(a_3)\varepsilon_B(b_2)a_1a_2\ot b_3\\
&+\varepsilon_A(a_2)\varepsilon_A(a_3)a_1\ot b_1b_2b_3 +\lambda\varepsilon_A(a_2)\varepsilon_B(b_1)\varepsilon_A(a_3)a_1\ot b_2b_3.
\end{align*}
On the other hand,
\begin{align*}
&\ (a_1\ot b_1)\cm \Big((a_2\ot b_2) \cm (a_3\ot b_3)\Big)\\
=&\ (a_1\ot b_1)\cm\Big(\varepsilon_B(b_2)a_2a_3\ot b_3+\varepsilon_A(a_3)a_2\ot b_2b_3+\lambda \varepsilon_A(a_3)\varepsilon_B(b_2)a_2\ot b_3\Big)\quad (\text{by Eq.~(\ref{eq:Tmul})})\\
=&\ \varepsilon_B(b_2)(a_1\ot b_1)\cm (a_2a_3\ot b_3)+\varepsilon_A(a_3)(a_1\ot b_1)\cm(a_2\ot b_2b_3)\\
&+\lambda \varepsilon_A(a_3)\varepsilon_B(b_2)(a_1\ot b_1)\cm (a_2\ot b_3)\\
=&\ \varepsilon_B(b_2) \Big(\varepsilon_B(b_1)a_1a_2a_3\ot b_3+\varepsilon_A(a_2a_3)a_1\ot b_1b_3+\lambda \varepsilon_A(a_2a_3)\varepsilon_B(b_1)a_1 \ot b_3\Big)\\
&+\varepsilon_A(a_3) \Big(\varepsilon_B(b_1)a_1a_2\ot b_2 b_3+\varepsilon_A(a_2)a_1\ot b_1b_2b_3+\lambda \varepsilon_A(a_2)\varepsilon_B(b_1)a_1 \ot b_2b_3\Big)\\
&+ \lambda \varepsilon_A(a_3)\varepsilon_B(b_2) \Big(\varepsilon_B(b_1)a_1a_2\ot b_3+\varepsilon_A(a_2)a_1\ot b_1b_3+\lambda \varepsilon_A(a_2)\varepsilon_B(b_1)a_1 \ot b_3\Big)\\
=&\ \varepsilon_B(b_2)\varepsilon_B(b_1)a_1a_2a_3\ot b_3-\lambda\varepsilon_B(b_2)\varepsilon_A(a_2)\varepsilon_A(a_3)a_1\ot b_1 b_3 \\ &-\lambda^2\varepsilon_B(b_2)\varepsilon_A(a_2)\varepsilon_A(a_3)\varepsilon_B(b_1)a_1\ot b_3+\varepsilon_A(a_3)\varepsilon_B(b_1)a_1a_2\ot b_2b_3\\
&+\varepsilon_A(a_3)\varepsilon_A(a_2)a_1\ot b_1b_2b_3+\lambda\varepsilon_A(a_3)\varepsilon_A(a_2)\varepsilon_B(b_1)a_1\ot b_2b_3\\
&+\lambda\varepsilon_A(a_3)\varepsilon_B(b_2)\varepsilon_B(b_1)a_1a_2\ot b_3
+\lambda\varepsilon_A(a_3)\varepsilon_B(b_2)\varepsilon_A(a_2)a_1\ot b_1 b_3\\
&+\lambda^2\varepsilon_A(a_3)\varepsilon_B(b_2)\varepsilon_A(a_2)\varepsilon_B(b_1)a_1\ot b_3\quad (\text{by Eq.~(\ref{eq:Augid})})\\
=&\ \varepsilon_B(b_2)\varepsilon_B(b_1)a_1a_2a_3\ot b_3
+\varepsilon_A(a_3)\varepsilon_B(b_1)a_1a_2\ot b_2b_3
+\varepsilon_A(a_3)\varepsilon_A(a_2)a_1\ot b_1b_2b_3\\
&+\lambda\varepsilon_A(a_3)\varepsilon_A(a_2)\varepsilon_B(b_1)a_1\ot b_2b_3
+\lambda\varepsilon_A(a_3)\varepsilon_B(b_2)\varepsilon_B(b_1)a_1a_2\ot b_3,
\end{align*}
as desired.

We now show that $A\ot B$ is an augmented algebras of weight $\lambda$. For any $a_1\ot b_1, a_2\ot b_2\in A\ot B$,
\begin{align*}
&\ \varepsilon_{A\ot B}\Big((a_1\ot b_1)\cm (a_2\ot b_2)\Big)\\
=&\ \varepsilon_{A\ot B}\Big(\varepsilon_B(b_1)a_1a_2\ot b_2+\varepsilon_A(a_2)a_1\ot b_1b_2+\lambda \varepsilon_A(a_2)\varepsilon_B(b_1)a_1\ot b_2\Big)\quad (\text{by Eq.~(\ref{eq:Tmul})})\\
=&\ \varepsilon_B(b_1)\varepsilon_{A\ot B}(a_1a_2\ot b_2)+\varepsilon_A(a_2)\varepsilon_{A\ot B}(a_1\ot b_1b_2)
+\lambda\varepsilon_A(a_2)\varepsilon_B(b_1)\varepsilon_{A\ot B}(a_1\ot b_2)\\
=&\ \varepsilon_B(b_1)\varepsilon_{A}(a_1a_2) \varepsilon_B(b_2)+\varepsilon_A(a_2)\varepsilon_{A}(a_1) \varepsilon_B (b_1b_2)\\
&+\lambda\varepsilon_A(a_2)\varepsilon_B(b_1)\varepsilon_{A}(a_1) \varepsilon_B(b_2)\quad (\text{by Eq.~(\ref{eq:Taug})})\\
=&-\lambda\varepsilon_B(b_1)\varepsilon_{A}(a_1)\varepsilon_{A}(a_2) \varepsilon_B(b_2)-\lambda\varepsilon_A(a_2)\varepsilon_{A}(a_1) \varepsilon_B (b_1)\varepsilon_B(b_2)\\
&+\lambda\varepsilon_A(a_2)\varepsilon_B(b_1)\varepsilon_{A}(a_1) \varepsilon_B(b_2)\quad (\text{by Eq.~(\ref{eq:Augid})})\\
=&-\lambda\varepsilon_{A}(a_1)\varepsilon_B(b_1)\varepsilon_{A}(a_2) \varepsilon_B(b_2)\\
=&-\lambda\varepsilon_{A\ot B}(a_1\ot b_1)\varepsilon_{A\ot B}(a_2\ot b_2).
\end{align*}
This completes the proof.
\end{proof}

Let $(A, m,1,\Delta)$ be an $\epsilon$-unitary bialgebra of weight $\lambda$.
Then the \bfk-module $A\ot A$ becomes a coalgebra with the coproduct~\cite[Example 2.2.2]{Gub}:
\begin{align*}
\Delta_{A\ot A}:A\ot A \overset{\Delta\ot \Delta}{\rightarrow}A\ot A\ot A\ot A\overset{\id_A\ot \tau \ot \id_A}{\rightarrow}A\ot A\ot A\ot A,
\end{align*}
where $\tau:A\ot A\rightarrow A\ot A, a\ot b\mapsto b\ot a$ is the switch map.
However the multiplicaiton $m: A\ot A\to A$ is not a morphism of coalgebras under this classical coproduct.
In fact,
\begin{align*}
\Delta\circ m(a\ot b)=&\ \Delta(ab)= a\cdot\Delta(b)+\Delta(a)\cdot b+\lambda (a\ot b)\\
=&\ \sum_{(b)} ab_{(1)} \ot b_{(2)}+\sum_{(a)} a_{(1)} \ot a_{(2)}b+\lambda (a\ot b) \\
\neq &\ \sum_{(a)}\sum_{(b)} a_{(1)}  b_{(1)} \ot a_{(2)} b_{(2)}\\
=& \ (m\ot m)\circ\Delta_{A\ot A}(a\ot b).
\end{align*}
We erect a new coproduct on $A\ot A$ to solve this incompatibility,
which generalizes the one in case of weight zero~\cite[Lemma~3.5]{MA}. Recall that an element $e\in A$ is called a {\bf group like element of weight $\lambda$} if it satisfies $\Delta(e)=\lambda(e\ot e)$~\cite{ZCGL18}.
By Remark~\ref{remk:units}, the unit in an $\epsilon$-unitary bialgebra is a group like element of weight $-\lambda$.

\begin{prop}\label{prop:col}
Let $(A, \Delta_A)$ and $(B, \Delta_B)$ be coassociative coalgebras (possibly without counit), with group like elements $1_A$ and $1_B$ of weight $-\lambda$. Then $A\ot B$ is a coassociative coalgebra with the coproduct defined by
\begin{align}
\Delta(a\ot b):=\sum_{(a)}(a_{(1)}\ot 1_B)\ot (a_{(2)}\ot b)+\sum_{(b)}(a\ot b_{(1)})\ot (1_A \ot b_{(2)})+\lambda(a\ot 1_B)\ot (1_A\ot b)
\label{eq:te}
\end{align}
for any $a\ot b \in  A \ot B$.
\end{prop}

\begin{proof}
It is sufficient to check the coassociative law:
\begin{align*}
(\mathrm{id} \ot \Delta)\Delta(a\ot b)=(\Delta \ot \mathrm{id}  )\Delta(a\ot b) \quad \text{ for } a\ot b \in A\ot B.
\end{align*}
By Eq.~(\ref{eq:te}) and $\Delta_A(1_A) = -\lambda 1_A\ot 1_A$ in Remark~\ref{remk:units},
\begin{equation}
\begin{aligned}
\Delta(1_A\ot b)&=(-\lambda\ot 1_B)\ot(1_A\ot b)+\sum_{(b)}(1_A\ot b_{(1)})\ot (1_A \ot b_{(2)})+\lambda(1_A\ot 1_B)\ot (1_A\ot b)\\
&=\sum_{(b)}(1_A\ot b_{(1)})\ot (1_A \ot b_{(2)}).
\end{aligned}
\label{eq:newco1}
\end{equation}
Similarly,
\begin{align}
\Delta(a\ot 1_B)=\sum_{(a)}(a_{(1)}\ot 1_B)\ot (a_{(2)}\ot 1_B).
\label{eq:newco2}
\end{align}
Now on the one hand,
\allowdisplaybreaks{
\begin{align*}
&\ (\mathrm{id} \ot \Delta)\Delta(a\ot b)\\
=&\ (\mathrm{id} \ot \Delta)\bigg(\sum_{(a)}(a_{(1)}\ot 1_B)\ot (a_{(2)}\ot b)+\sum_{(b)}(a\ot b_{(1)})\ot (1_A \ot b_{(2)})
+\lambda(a\ot 1_B)\ot (1_A\ot b)\bigg)\\
=&\ \sum_{(a)}(a_{(1)}\ot 1_B)\ot\Delta (a_{(2)}\ot b)+\sum_{(b)}(a\ot b_{(1)})\ot \Delta(1_A \ot b_{(2)})
+\lambda(a\ot 1_B)\ot \Delta(1_A\ot b)\\
=&\ \sum_{(a)}(a_{(1)}\ot 1_B)\ot \bigg(\sum_{(a_{(2)})}(a_{(2)(1)}\ot 1_B)\ot (a_{(2)(2)}\ot b)+\sum_{(b)}(a_{(2)}\ot b_{(1)})\ot (1_A \ot b_{(2)})\\
&\ +\lambda(a_{(2)}\ot 1_B)\ot (1_A\ot b)\bigg)+\sum_{(b)}(a\ot b_{(1)})\ot \bigg(\sum_{(b_{(2)})}(1_A\ot b_{(2)(1)})\ot (1_A \ot b_{(2)(2)})\bigg)\\
&\ +\lambda(a\ot 1_B)\ot \bigg(\sum_{(b)}(1_A\ot b_{(1)})\ot (1_A \ot b_{(2)})\bigg)   \quad \text{(by Eqs.~(\ref{eq:te}) and~(\ref{eq:newco1}))}\\
=&\ \sum_{(a)} (a_{(1)}\ot 1_B)\ot (a_{(2)}\ot 1_B)\ot (a_{(3)}\ot b)
+\sum_{(a)}\sum_{(b)}(a_{(1)}\ot 1_B)\ot(a_{(2)}\ot b_{(1)})\ot (1_A \ot b_{(2)})\\
&\ +\lambda\sum_{(a)}(a_{(1)}\ot 1_B)\ot(a_{(2)}\ot 1_B)\ot (1_A\ot b)
+\sum_{(b)} (a\ot b_{(1)})\ot (1_A\ot b_{(2)})\ot (1_A \ot b_{(3)})\\
&\ +\lambda\sum_{(b)}(a\ot 1_B)\ot (1_A\ot b_{(1)})\ot (1_A \ot b_{(2)}) .
\end{align*}
}
On the other hand,
\allowdisplaybreaks{
\begin{align*}
&\ (\Delta \ot\mathrm{id})\Delta(a\ot b)\\
=&\ (\Delta \ot\mathrm{id})\bigg(\sum_{(a)}(a_{(1)}\ot 1_B)\ot (a_{(2)}\ot b)+\sum_{(b)}(a\ot b_{(1)})\ot (1_A \ot b_{(2)})
+\lambda(a\ot 1_B)\ot (1_A\ot b)\bigg)\\
=&\ \sum_{(a)}\Delta (a_{(1)}\ot 1_B)\ot(a_{(2)}\ot b)+\sum_{(b)}\Delta(a\ot b_{(1)})\ot (1_A \ot b_{(2)})
+\lambda\Delta(a\ot 1_B)\ot (1_A\ot b)\\
=&\ \sum_{(a)}\sum_{(a_{(1)})}(a_{(1)(1)}\ot 1_B)\ot (a_{(1)(2)}\ot 1_B)\ot (a_{(2)}\ot b)+\sum_{(b)}\bigg(\sum_{(a)}(a_{(1)}\ot 1_B)\ot (a_{(2)}\ot b_{(1)})\\
&\ +\sum_{(b_{(1)})}(a\ot b_{(1)(1)})\ot (1_A \ot b_{(1)(2)})+\lambda(a\ot 1_B)\ot (1_A\ot b_{(1)})\bigg)\ot (1_A \ot b_{(2)})\\
&\ +\lambda\sum_{(a)}(a_{(1)}\ot 1_B)\ot (a_{(2)}\ot 1_B) \ot (1_A\ot b)  \quad \text{(by Eqs.~(\ref{eq:te}) and~(\ref{eq:newco2}))}\\
=&\ \sum_{(a)} (a_{(1)}\ot 1_B)\ot (a_{(2)}\ot 1_B)\ot (a_{(3)}\ot b)+\sum_{(b)}\sum_{(a)}(a_{(1)}\ot 1_B)\ot (a_{(2)}\ot b_{(1)})\ot(1_A \ot b_{(2)})\\
&\ +\sum_{(b)} (a\ot b_{(1)})\ot (1_A \ot b_{(2)})\ot (1_A \ot b_{(3)})+ \lambda\sum_{(b)}(a\ot 1_B)\ot (1_A\ot b_{(1)})\ot (1_A \ot b_{(2)})\\
&\ +\lambda\sum_{(a)}(a_{(1)}\ot 1_B)\ot (a_{(2)}\ot 1_B) \ot (1_A\ot b)\\
=&\  (\mathrm{id} \ot \Delta)\Delta(a\ot b).
\end{align*}
}
This completes the proof.
\end{proof}

Then we have the following results which  generalise
simultaneously the one introduced by Aguiar~\cite{MA} and the one initiated by Foissy~\cite{Foi09}.
\begin{theorem}\label{them:dual}
We have the following two dual statements.
\begin{enumerate}
\item Let $(A, \mu_A, \Delta_A, \varepsilon_A)$ be an $\epsilon$-counitary bialgebra of weight $\lambda$ and view $(A\ot A, \cdot_{\varepsilon})$ as an algebra as in Proposition~\ref{prop:aug}. Then $\Delta_A: (A, \mu_A)\rightarrow (A\ot A, \cdot_{\varepsilon})$ is a morphism of algebras.
    \label{it:a}
\item Let $(A, \mu_A, \Delta_A, 1_A)$ be an $\epsilon$-unitary bialgebra of weight $\lambda$ and view $(A\ot A, \Delta)$ as a coalgebra as in Proposition~\ref{prop:col}.
Then $\mu_{A\ot A}: (A\ot A, \Delta) \rightarrow (A, \Delta_A)$ is a morphism of coalgebras. \label{it:b}
\end{enumerate}
\end{theorem}

\begin{proof}
(\ref{it:a})
It suffices to prove
\begin{align*}
\Delta(a)\cdot_{\varepsilon} \Delta(b)=\Delta(ab) \, \text{ for}\, a, b \in A,
\end{align*}
which follows from
\begin{align*}
\Delta_A(a)\cdot_{\varepsilon} \Delta_A(b)=&\left(\sum_{(a)}a_{(1)} \ot a_{(2)}\right) \cdot_{\varepsilon}\left(\sum_{(b)}b_{(1)} \ot b_{(2)}\right)\\
=&\sum_{(a)}\sum_{(b)}\left(a_{(1)} \ot a_{(2)}\right) \cdot_{\varepsilon}\left(b_{(1)} \ot b_{(2)}\right)\\
=&\sum_{(a)}\sum_{(b)}\Big(\varepsilon(a_{(2)})a_{(1)}b_{(1)}\ot b_{(2)}+\varepsilon(b_{(1)})a_{(1)}\ot a_{(2)}b_{(2)}+\lambda \varepsilon(a_{(2)})\varepsilon(b_{(1)})a_{(1)} \ot b_{(2)} \Big)\\
=&\sum_{(b)}ab_{(1)}\ot b_{(2)}+\sum_{(a)}a_{(1)}\ot a_{(2)}b+\lambda (a\ot b)\quad(\text{by the counicity})\\
=&\ a\cdot \Delta_A(b)+\Delta_A(a)\cdot b+\lambda (a\ot b)\\
=&\Delta_A(ab).
\end{align*}
(\ref{it:b}) It is the dual of Item~(\ref{it:a}).
\end{proof}

\subsection{The dual of weighted infinitesimal bialgebras}
In this subsection, we propose the concept of a weighted coderivation and prove that the dual of an $\epsilon$-bialgebra of weight $\lambda$ is also an $\epsilon$-bialgebra of weight $\lambda$.

\begin{defn}\cite{ZGZ18}
Let $\lambda$ be a given element of $\bfk$, and let $(A, \mu)$ be an algebra and $A\ot A$ an $(A, A)$-bimodule.
Then a linear map $D:A\rightarrow A\ot A$ is called a {\bf derivation of weight $\lambda$} of $A$ if it satisfies Eq.~(\ref{eq:wib}), that is,
\begin{align*}
D \mu = s (\id_A \ot D)+ t(D\ot \id_A)+\lambda (\id_A \ot \id_A),
\end{align*}
where $s: A\ot (A\ot A)\rightarrow A\ot A $ and $t:(A\ot A) \ot A \rightarrow A\ot A$ are the bimodule structure maps.
\end{defn}

Dually, we propose the concept of a weighted coderivation.

\begin{defn}\label{def:coder}
Let $\lambda$ be a given element of $\bfk$, and let $C$ be a coalgebra and $C\ot C$ a $(C, C)$-bicomodule.
Then a map $D:C\ot C \rightarrow C$ is called a {\bf coderivation of weight $\lambda$} if it satisfies
\begin{align*}
\Delta D =(\id_C\ot D)s + (D\ot \id_C)t +\lambda(\id_C\ot \id_C),
\end{align*}
where $s:C\ot C \rightarrow C\ot (C\ot C) $ and $t:C\ot C \rightarrow (C\ot C)\ot C$ are the bicomodule structure maps.
Here we view $C\ot C$ as a $(C,C)$-bicomodule via
\begin{equation}
s=\Delta\ot \id_C\, \text{ and }\, t=\id_C\ot \Delta.
\label{eq:cdota}
\end{equation}
\end{defn}

\begin{remark}
\begin{enumerate}
\item The classical derivation is a derivation of weight zero and the classical coderivation~\cite{Aguu02} is coderivation of weight zero.
\item We emphasize that the main difference between classical bialgebras and weighted $\epsilon$-bialgebras is the compatibility condition about the coproduct $\Delta$ and the multiplication $\mu$. More precisely, $\Delta$ is an algebra morphism in classical bialgebras, however, $\Delta$ is a weighted derivation in weighted $\epsilon$-bialgebras.
\end{enumerate}
\end{remark}

In classical bialgebras, the coproduct $\Delta$ is a morphism of algebras, which is equivalent to the  multiplication $\mu$ is a morphism of coalgebras. In weighted $\epsilon$-bialgebras, an analogous result is given by the following proposition.

\begin{prop}\label{prop:dual1}
Let $(A, \mu)$ be an algebra and $(A, \Delta)$ a coalgebra. Then the following assertions are equivalent:
\begin{enumerate}
\item $\Delta: A\to A\ot A$ is a derivation of weight $\lambda$.
\item  $\mu: A\ot A \to A$ is a coderivation of weight $\lambda$.
\end{enumerate}
Here $A\ot A$ is viewed as an $(A, A)$-bimodule and an $(A, A)$-bicomodule
by Eq.~(\ref{eq:dota}) and Eq.~(\ref{eq:cdota}), respectively.
\end{prop}

\begin{proof}
Note that the compatibility condition in Eq.~(\ref{eq:wib}) can be written as
\begin{align}
\Delta \mu =(\mu\ot \id_A)(\id_A \ot \Delta)+(\id_A\ot \mu)(\Delta\ot \id_A)+\lambda (\id_A \ot \id_A).
\label{eq:cocycle1}
\end{align}
Eq.~(\ref{eq:cocycle1}) implies that $\Delta: A\rightarrow A\ot A$ is a weighted derivation of the algebra $(A, \mu)$ with values in the $A$-bimodule $A\ot A$, in other words, $\mu: A\ot A\rightarrow A$ is a weighted coderivation from the $A$-bicomodule $A\ot A$ with values in the coalgebra $(A, \Delta)$.
\end{proof}

\begin{theorem}\label{thm:dual}
Let $(A, \mu, \Delta)$ be a $\epsilon$-bialgebra of weight $\lambda$ with $A^*\ot A^*\cong (A \ot A)^*$ as modules.
Then the dual space $A^*$ is an $\epsilon$-bialgebra of weight $\lambda$ with the multiplication
\begin{align*}
A^*\ot A^*\cong (A \ot A)^*\stackrel{\Delta^*}{\longrightarrow} A^*
\end{align*}
and the coproduct
\begin{align*}
A^*\stackrel{m^*}{\longrightarrow}(A\ot A)^*\cong A^*\ot A^*.
\end{align*}
\end{theorem}
\begin{proof}
By Proposition~\ref{prop:dual1}, the derivation of weight $\lambda$ and the coderivation of weight $\lambda$ correspond to each other by duality. Then the result follows from the classical finite dual between algebras and coalgebras.
\end{proof}

\begin{remark}
\begin{enumerate}
\item Let $(A, \mu, \Delta)$ be an $\epsilon$-bialgebra of weight $\lambda$. Then $(A, -\mu, \Delta)$ and $(A, \mu, -\Delta)$ are two $\epsilon$-bialgebras of weight $-\lambda$, and $(A, -\mu, -\Delta)$ is an $\epsilon$-bialgebra of weight $\lambda$.

\item Let $(A, \mu, \Delta)$ be an $\epsilon$-bialgebra of weight zero.  Let $\mu^{op}=\mu\tau$, $\Delta^{cop}=\tau\Delta$. Then $(A, \mu^{op}, \Delta^{cop})$ is an $\epsilon$-bialgebra of weight zero~\cite{Aguu02}.
\end{enumerate}
\end{remark}
\section{Weighted infinitesimal Hopf modules}\label{sec:wei}
In this section, we introduce the concept of weighted infinitesimal Hopf modules. We show that any $A$-module carries a natural structure of weighted $\epsilon$-unitary Hopf module over $A$ when $A$ is a weighted quasitriangular $\epsilon$-unitary bialgebra.

\subsection{The basic definitions and  examples}

Now we propose the concept of a weighted infinitesimal Hopf module.
\begin{defn}\label{def:ihmodule}
Let $\lambda$ be a given element of $\bfk$ and $(A, \mu, \Delta)$  an $\epsilon$-bialgebra of weighted $\lambda$. A { \bf (left) infinitesimal Hopf module}  of {\bf weight $\lambda$} is a $\bfk$-module $M$ equipped with a left $A$-module structure $\gamma:A\ot M\rightarrow M, a\ot m\mapsto am$ and a left $A$-comodule structure $\Lambda:M\rightarrow A\ot M, m\mapsto \sum_{(m)}m_{(-1)}\ot m_{(0)}$, such that
\begin{align}
\Lambda \gamma= (\mu \ot \id_M)(\id_A \ot \Lambda)+(\id_A \ot \gamma)(\Delta \ot \id_M)+\lambda (\id_A \ot \id_M), \label{eq:Hmodule}
\end{align}
that is
\begin{align}
\Lambda(am)=a\Lambda(m)+\Delta(a)m+\lambda (a\ot m). \label{eq:Hmodule1}
\end{align}
If further $(A, \mu, 1, \Delta)$  is an $\epsilon$-unitary bialgebra of weighted $\lambda$, then $M$ is called a {\bf (left) infinitesimal unitary Hopf module}  {\bf of weight $\lambda$}.
\end{defn}
\begin{remark}
\begin{enumerate}

\item In terms of above notations, the compatibility condition of weighted $\epsilon$-Hopf module given in Eq.~(\ref{eq:Hmodule1}) may be written as
\begin{align*}
\sum_{(m)}(am)_{(-1)}\ot (am)_{(0)}=\sum_{(m)}am_{(-1)}\ot m_{(0)}+\sum_{(a)}a_{(1)}\ot a_{(2)}\cdot m+\lambda a\ot m,
\end{align*}
which is different from the compatibility condition of classical Hopf module
\begin{align*}
\sum_{(m)}(am)_{(-1)}\ot (am)_{(0)}=\sum_{(m), (a)}a_{(1)}m_{(-1)}\ot a_{(2)}\cdot m_{(0)}.
\end{align*}

\item The left infinitesimal Hopf module studied by Aguiar~\cite{Aguu02} is a left infinitesimal Hopf module of weight zero.
\end{enumerate}
\end{remark}
The weighted $\epsilon$-Hopf module admits an analogy to the classical Hopf module over the ordinary Hopf algebras. The basic examples of classical Hopf modules adapted from~\cite[Section 1.9]{Mon93} also admit the following weighted infinitesimal Hopf versions in the context of weighted $\epsilon$-bialgebras. We refer to~\cite{Aguu02} for the results in the case of weight zero.

\begin{exam}\label{exam:exmo}
Let $\lambda$ be a given element of $\bfk$ and $(A, \mu, \Delta)$  an $\epsilon$-bialgebra of weighted $\lambda$.
\begin{enumerate}
\item $A$ itself is an $\epsilon$-Hopf module of weight $\lambda$ by taking $\gamma=\mu$ and $\Lambda=\Delta$, according to the definition of $\epsilon$-bialgebra of weight $\lambda$.
\item \label{it:b} Let $V$ be a vector space. Then $A\ot V$ is an  $\epsilon$-Hopf module of weight $\lambda$ by setting
\begin{align}
\gamma=\mu\ot \id_V: A\ot A \ot V\rightarrow A\ot V \, \text{ and }\, \Lambda=\Delta\ot \id_V: A\ot V\rightarrow A\ot A\ot V.
\label{eq:exam1}
\end{align}
In fact, for any $a\ot v\in A\ot V$, we have
\begin{align*}
(\id_A \ot \Lambda)\Lambda(a\ot v)=&\ (\id_A \ot (\Delta\ot \id_V))(\Delta\ot \id_V)(a\ot v)\quad (\text{by Eq.~(\ref{eq:exam1})})\\
=&\ (\id_A \ot (\Delta\ot \id_V))(\Delta(a)\ot v)
=\sum_{(a)}(\id_A \ot (\Delta\ot \id_V))(a_{(1)}\ot a_{(2)} \ot v)\\
=&\sum_{(a)}a_{(1)}\ot \Delta (a_{(2)}) \ot v
=\sum_{(a)}\Delta (a_{(1)})\ot a_{(2)} \ot v  \ (\text{by the coassociative law})\\
=&\sum_{(a)}(\Delta\ot \id_{A\ot V}) (a_{(1)}\ot a_{(2)} \ot v)=(\Delta\ot \id_{A\ot V}) (\Delta(a) \ot v)\\
=&\ (\Delta\ot \id_{A\ot V})\Lambda(a\ot v).
\end{align*}
Then the coassociative law is desired. Next we can check that
\begin{align*}
\Lambda \gamma(a\ot b \ot v)=&\ (\Delta\ot \id_V) (\mu\ot \id_V) (a\ot b \ot v)=\Delta(ab)\ot v\quad (\text{by Eq.~(\ref{eq:exam1})})\\
=&\ \left(a\cdot \Delta(b)+\Delta(a)\cdot b+\lambda (a\ot b)\right)\ot v \quad (\text{by Eq.~(\ref{eq:wib})})\\
=&\ (\mu \ot  \id_{A\ot V})(\id_A \ot \Delta \ot \id_v)(a\ot b\ot v)\\
&+(\id_A\ot \mu \ot \id_V)(\Delta  \ot \id_{A\ot V})(a\ot b\ot v)
+\lambda (\id_A \ot \id_{A\ot V})(a\ot b\ot v)\\
=&\ (\mu \ot \id_{A\ot V})(\id_A \ot \Lambda)(a\ot b\ot v)
+(\id_A\ot \gamma)(\Delta  \ot \id_{A\ot V})(a\ot b\ot v)\\
&+\lambda (\id_A \ot \id_{A\ot V})(a\ot b\ot v).
\end{align*}
Thus the compatibility condition of weighted $\epsilon$-Hopf module given in Eq.~(\ref{eq:Hmodule}) is satisfied.
\end{enumerate}
\end{exam}

\subsection{Modules and weighted infinitesimal Hopf modules}
In this subsection, we show that when $A$ is a weighted quasitriangular $\epsilon$-unitary bialgebra, any $A$-module carries a natural structure of weighted $\epsilon$-unitary Hopf module over $A$.
We first give some basic definitions and notations  that will be used in this section. We refer to~\cite[Section~5]{MA}
for the classical results in the case of $\lambda=0$.

\begin{defn} Let $\lambda$ be a given element of $\bfk$, and let $A$ be a unitary algebra and $W$ an $(A, A)$-bimodule.
\begin{enumerate}
\item A linear map $\Delta_{r} :A\rightarrow W$ associated to an element $r\in W$ is called {\bf principal} if it satisfies
    \begin{align}
    \Delta_{r}(a) :=a\cdot r- r\cdot a-\lambda (a\ot 1) \, \text{ for }\, a\in A. \label{eq:prin}
    \end{align}

\item An element $r\in W$ is called {\bf $A$-invariant} if it satisfies
\begin{align*}
a\cdot r=r\cdot a \, \text{ for }\, a\in A.
\end{align*}
\end{enumerate}
\end{defn}

For an algebra $A$, $A\ot A\ot A$ is viewed as an $A$-bimodule via
\begin{align*}
a\cdot (b\ot c\ot d)=ab\ot c\ot d \, \text{ and }\, (b\ot c\ot d)\cdot a =b\ot c\ot da,
\end{align*}
where $a, b, c, d\in A$.

\begin{remark}
 The $\Delta_{r}$ defined by Eq.~(\ref{eq:prin}) is a derivation of weight $\lambda$.
\end{remark}

 We now recall the definition of a weighted quasitriangular $\epsilon$-unitary bialgebra.
Let $A$ be a unitary algebra and $r=\sum_i u_i \ot v_i\in A\ot A$.
Then the principle derivation of weight $\lambda$
$$\Delta_r: A\rightarrow A\ot A,\quad a\mapsto a\cdot r - r\cdot a - \lambda (a\ot 1)$$
is coassociative if and only if the element
\begin{align*}
r_{13}r_{12}-r_{12}r_{23}+r_{23}r_{13}-\lambda r_{13}\in A\ot A\ot A
\end{align*}
is $A$-invariant~\cite[Theorem~4.5]{ZGZ18}. Having this fact in hand, our previous work shows that that the quadruple $(A, \mu, 1, \Delta_r)$ is an $\epsilon$-unitary bialgebra of weight $\lambda$ if $r$ is a solution of an AYBE of weight $\lambda$ in $A$~\cite[Theorem~4.6]{ZGZ18}. We refer to the
quadruple $(A,\mu,1,r)$ as a weighted quasitriangular $\epsilon$-unitary bialgebra, more precisely,

\begin{defn}\cite[Definition 5.1]{ZGZ18}
Let $(A, \mu, 1)$ be a unitary algebra. A {\bf quasitriangular infinitesimal unitary bialgebra} {\bf of weight $\lambda$} is a quadruple $(A, m, 1, r)$ consisting of a unitary algebra  $(A, m, 1)$ and a solution $r\in A\ot A$  of an associative Yang-Baxter equation of weight $\lambda$.
\label{def:qua}
\end{defn}

Recall that $\Delta_{r}$ is defined in Eq.~(\ref{eq:prin}) for a $r\in A\ot A$.

\begin{remark}
\begin{enumerate}
\item The quadruple $(A, \mu, 1,\Delta_r)$ is indeed an $\epsilon$-unitary bialgebra of weight $\lambda$.

\item A quasitriangular $\epsilon$-bialgebra studied in~\cite{MA} is a quasitriangular $\epsilon$-unitary bialgebra of weight zero. In this case, the quadruple $(A, \mu, 1, \Delta_r)$ is  an $\epsilon$-unitary bialgebra of weight zero.
\item \cite{ZGZ18} Let $(A, \mu, 1, r)$ be a quasitriangular $\epsilon$-unitary bialgebra of weight $\lambda$ with $r=\sum_{i}u_i\ot v_i\in A\ot A$. Define  two binary operations $\succ, \prec$ on $A$ by
\begin{align*}
a\succ b:=\sum_{i}u_iav_ib \, \text { and }\, a\prec b:=\sum_{i}au_ibv_i-\lambda ab.
\end{align*}
Then the triple $(A, \succ, \prec )$ is a dendriform algebra.
\end{enumerate}
\end{remark}

We need the following lemma.

\begin{lemma}\cite[Proposition 5.3]{ZGZ18}
Let $(A, \mu, 1,  r)$ be a quasitriangular $\epsilon$-unitary bialgebra of weight $\lambda$ and $\Delta :=\Delta_r$. Then
\begin{align}\label{eq:a}
&\Delta(a)=a\cdot r-r\cdot a-\lambda (a\ot 1),\\ \label{eq:b}
&(\Delta\ot \id )(r)=-r_{23}r_{13}, \, \text{ and }\, \\ \label{eq:c}
&(\id \ot \Delta)(r)=r_{13}r_{12}-\lambda (r_{13}+r_{12}).
\end{align}
Conversely, if an $\epsilon$-unitary bialgebra $(A, \mu, 1,\Delta)$ of weight $\lambda$ satisfies Eqs.~(\ref{eq:a}), (\ref{eq:b}) and (\ref{eq:c})
for some $r = \sum_i u_i\ot v_i \in A\ot A$, then $(A, \mu, 1, r)$ is a quasitriangular $\epsilon$-unitary bialgebra of weight $\lambda$ and $\Delta=\Delta_r$.
\label{lem:quaiff}
\end{lemma}

We now state our main result in this subsection.
\begin{theorem}\label{thm:main}
Let $(A, \mu, 1, r)$ be a quasitriangular $\epsilon$-unitary bialgebra of weight $\lambda$ and $M$ a left $A$-module. Then $M$ becomes a left $\epsilon$-unitary Hopf module of weight $\lambda$ over $A$ with the $\Lambda: M \rightarrow A \ot M$ given by
\begin{align*}
\Lambda(m):=-\sum_{i}u_i \ot v_im \, \text{ for }\, m\in M.
\end{align*}
\end{theorem}

\begin{proof}
We first prove the coassociative law:
\begin{align*}
(\id_A \ot \Lambda)\Lambda (m)=(\Delta_r\ot \id_A) \, \text{ for }\, m\in M.
\end{align*}
Let $r=\sum_{i}u_i \ot v_i$. Then by Lemma~\ref{lem:quaiff}, we have
\begin{align*}
(\Delta_r \ot \id_A)(r)=\sum_{i}\Delta_r(u_i)\ot v_i=-r_{23}r_{13}=-\sum_{i,j}u_i\ot u_j\ot v_jv_i.
\end{align*}
Thus
\begin{align*}
(\Delta_r \ot \id_A)\Lambda (m)=&-\sum_{i}\Delta_r(u_i)\ot v_im=\sum_{i,j}u_i\ot u_j\ot v_jv_im\\
=&-\sum_{i}u_i \ot \Lambda(v_im)=(\id_A \ot \Lambda)\Lambda (m).
\end{align*}
We then check the compatibility condition of the weighted $\epsilon$-unitary Hopf module in Eq.~(\ref{eq:Hmodule1}). Since $\Delta_{r}(a)=a\cdot r- r\cdot a-\lambda (a\ot 1)$, we have
\begin{align*}
a\Lambda(m)+\Delta_r(a)m=&-\sum_{i}au_i \ot v_im+\sum_{i}au_i \ot v_im-\sum_{i}u_i \ot v_iam-\lambda (a\ot m)\\
=&-\sum_{i}u_i \ot v_iam-\lambda (a\ot m)\\
=&\Lambda(am)-\lambda (a\ot m).
\end{align*}
This completes the proof.
\end{proof}

\section{Infinitesimal unitary bialgebras of  rooted forests}
\label{sec:infbi}
In this section, we first recall the concepts of planar rooted forests~\cite{Sta97} and decorated planar rooted forests~\cite{Foi02, Guo09}. We then give a new
way to decorate planar rooted forests that generalises the ones introduced and studied in~\cite{Foi02, Guo09, ZGG16}.
We also define a coproduct on our decorated planar rooted forests to equip them with a coalgebraic structure, with an eye toward constructing an $\epsilon$-unitary bialgebra on them.

\subsection{Decorated planar rooted forests}
\label{ss:dprfor}
A $\mathbf{rooted\ tree}$ is a finite graph, connected and without cycles, with a special vertex called the $\mathbf{root}$. A $\mathbf{planar\ rooted\ tree}$  is a rooted tree with a fixed embedding into the plane.
The first few planar rooted trees are listed below:
$$\tun,\ \tdeux,\ \ttroisun,\ \tquatreun,\ \ttroisdeux,\  \tquatretrois,\ \tquatredeux,\ \tquatrequatre,\ \tquatrecinq, \  \tcinqdeux,$$
where the root of a tree is on the bottom~\cite{Sta97}.
Let $\calt$ denote the set of planar rooted trees and $M(\calt)$ the free monoid generated by $\calt$ with the concatenation product, denoted by $\mul$ and usually suppressed. The empty tree in $M(\calt)$ is denoted by $\etree$.
An element in $M(\calt)$, called a {\bf planar rooted forest}, is a noncommutative concatenation of planar rooted trees, denoted by $F=T_1\cdots T_n$ with $T_1, \ldots, T_n \in \calt$. Here we use the convention that $F=\etree$ when $n=0$. The first few planar rooted forests are listed below:
$$\etree,\ \,\tun,\ \,\tun\tun,\ \,\tdeux,\ \, \tdeux\tun,\ \,\tun \tdeux,\ \, \tun\ttroisun,\ \,\tdeux\tun\tun,\ \,\tun \tdeux \tun,\ \,\ttroisun\ttroisdeux.$$

We now elaborate on a new decoration on planar rooted forests motivated from~\cite{ZGG16}.
Let $\Omega$ be a nonempty set and $X$ a set whose elements are not in $\Omega$.
Denote by
$\calt(X\sqcup \Omega)$ (resp.~$\calf(X\sqcup \Omega):=M(\calt(X\sqcup \Omega))$)
the set of planar rooted trees (resp.~forests) whose vertices (leaf and internal vertices) are decorated by elements of $X\sqcup \Omega$.

Let $\rts$ (resp.~$\rfs$) denote the subset of $\calt(X\sqcup \Omega)$ (resp.~$\calf(X\sqcup \Omega)$) consisting of vertex decorated planar rooted trees (resp. forests) whose internal vertices are decorated by elements of $\Omega$ exclusively and  leaf vertices are decorated by elements of $X\sqcup \Omega$. In other words, all internal vertices, as well as possibly some of the leaf vertices, are decorated by $\Omega$. If a tree has only one vertex, the vertex is treated as a leaf vertex. Here are some examples in $\rts$:
$$\tdun{$\alpha$},\ \, \tdun{$x$},\ \, \tddeux{$\alpha$}{$\beta$},\ \,  \tddeux{$\alpha$}{$x$}, \ \, \tdtroisun{$\alpha$}{$\beta$}{$\gamma$},\ \,\tdtroisun{$\alpha$}{$x$}{$\gamma$}, \ \,\tdtroisun{$\alpha$}{$x$}{$y$}, \ \, \tdquatretrois{$\alpha$}{$\beta$}{$\gamma$}{$\beta$},\ \, \tdquatretrois{$\alpha$}{$\beta$}{$\gamma$}{$x$}, \ \, \tdquatretrois{$\alpha$}{$\beta$}{$x$}{$y$};$$
whereas, the following are some examples not in $\rts$:
$$\tddeux{$x$}{$\alpha$},\ \,  \tddeux{$x$}{$y$}, \ \, \tdtroisun{$x$}{$\beta$}{$\alpha$}, \ \, \tdquatretrois{$\alpha$}{$x$}{$\gamma$}{$\beta$},$$
where $x, y \in X$ and $\alpha,\beta,\gamma\in \Omega$.

Let us emphasize that the decorated planar rooted forests $\hrts$ considered here are very general, which
include the undecorated planar rooted forests in the noncommutative Connes-Kreimer Hopf algebra~\cite{Hol03},
the ones in the Foissy-Holtkamp Hopf algebra~\cite{Foi02} and the ones in~\cite{ZGG16}. See Remark~\ref{re:3ex} below.

\begin{remark}
Now we give some special cases of our decorated planar rooted forests.
\begin{enumerate}
\item If $X=\emptyset$ and $\Omega$ is a singleton set, then all decorated planar rooted forests in $\mathcal{F}(X, \Omega)$ have the same decoration and so
can be identified with the planar rooted forests without decorations, which is the object studied in the well-known Foissy-Holtkamp Hopf algebra---the noncommutative version of Connes-Kreimer Hopf algebra~\cite{Foi02, Hol03}.

\item If $X = \emptyset$, then $\mathcal{F}(X, \Omega)$ is the object employed in the decorated Foissy-Holtkamp Hopf algebra~\cite{Foi02}. \label{it:2ex}

\item If $\Omega$ is a singleton set, then $\rfs$ was introduced and studied in~\cite{ZGG16} to construct a cocycle Hopf algebra on decorated planar rooted forests.
\item The rooted forests in $\rfs$, whose leaf vertices are decorated by elements of $X$ and internal vertices are decorated by elements of $\Omega$, are introduced in~\cite{Guo09}. However, this decoration can't deal with the unity and the algebraic structures on this decorated rooted forests are all nonunitary.
\end{enumerate}
\label{re:3ex}
\end{remark}

Define
\begin{align*}
\hrts:= \bfk \rfs=\bfk M(\rts)
\end{align*}
to be the free $\bfk$-module spanned by $\rfs$.
For each  $\omega\in \Omega$, let
$$B^+_\omega:\hrts\to \hrts$$
to be the linear grafting operation by taking $\etree$ to $\bullet_\omega$ and sending a rooted forest in $\hrts$ to its grafting with the new root decorated by $\omega$. For example,
$$B_{\omega}^{+}(\etree)=\tdun{$\omega$} \ ,\ \  B_{\omega}^{+}(\tdun{$x$}\tddeux{$\alpha$}{$y$})=\tdquatretrois{$\omega$}{$\alpha$}{$y$}{$x$},\ \  B_{\omega}^{+}(\tddeux{$\beta$}{$\alpha$}\tdun{$x$})=
\tdquatredeux{$\omega$}{$x$}{$\beta$}{$\alpha$},$$
where $\alpha, \beta, \omega\in \Omega$ and $x, y \in X$. Note that $\hrts$ is closed under the concatenation $m_{RT}$.

For $F=T_1\cdots T_{n}\in \rfs$ with $n\geq 0$ and $T_1,\cdots,T_{n}\in \rts$, we define $\bre(F):=n$
to be the {\bf breadth} of $F$. Here we use the convention that $\bre(\etree) = 0$ when $n=0$.
In order to define the depth of a decorated planar rooted forests, we build a recursive structure on $\rfs$.
Denote $\bullet_{X}:=\{\bullet_{x}\mid x\in X\}$ and set
\begin{align*}
\calf_0:=M(\bullet_{X})=S(\bullet_{X})\sqcup \{\etree\},
\end{align*}
where $M(\bullet_{X})$ (resp.~$S(\bullet_{X})$) is the submonoid (resp.~subsemigroup) of $\rfs$ generated by $\bullet_{X}$.
Here we are abusing notion slightly since $M(\bullet_{X})$ (resp.~$S(\bullet_{X})$) is also isomorphic to the free monoid
(resp. semigroup) generated by $\bullet_{X}$.
Suppose that $\calf_n$ has been defined for an $n\geq 0$, then define
\begin{align*}
\calf_{n+1}:=M(\bullet_{X}\sqcup (\sqcup_{\omega\in \Omega} B_{\omega}^{+}(\calf_n))).
\end{align*}
Thus we obtain $\calf_n\subseteq \calf_{n+1}$ and
\begin{align}
\rfs = \lim_{\longrightarrow} \calf_n=\bigcup_{n=0}^{\infty}\calf_n.
\label{eq:rdeff}
\end{align}
Now elements $F\in \calf_n\setminus \calf_{n-1}$ are said to have {\bf depth} $n$, denoted by $\dep(F) = n$.
Here are some examples:
\begin{align*}
\dep(\etree) =&\ \dep(\bullet_x) =0,\ \dep(\bullet_\omega)=\dep(B^+_{\omega}(\etree)) = 1,\  \dep(\tddeux{$\omega$}{$\alpha$})= \dep(B^+_{\omega}(B^+_{\alpha}(\etree))) =2, \\
\dep(\tdun{$x$}\tddeux{$\omega$}{$y$}\tdun{$y$}) =&\ \dep(\tddeux{$\omega$}{$y$})=
\dep(B^+_{\omega}(\bullet_y)) =1, \ \dep(\tdtroisun{$\omega$}{$x$}{$\alpha$}) = \dep(B^+_{\omega}(B^+_{\alpha}(\etree) \bullet_x)) = 2,
\end{align*}
where $\alpha,\omega\in \Omega$ and $x, y \in X$.

\subsection{A new coproduct on decorated planar rooted forests}
In this subsection, we construct a new coproduct $\col$ on $\hrts$, leading to an $\epsilon$-unitary bialgebra of weight zero on $\hrts$.
It suffices to define $\col(F)$ for basis elements $F\in \rfs$ by induction on $\dep(F)$.
For the initial step of $\dep(F)=0$, we define
\begin{equation}
\col(F) :=
\left\{
\begin{array}{ll}
0 & \text{ if } F = \etree, \\
\etree \ot \etree & \text{ if } F = \bullet_x \text{ for some } x \in X,\\
\bullet_{x_{1}}\cdot \col(\bullet_{x_{2}}\cdots\bullet_{x_{m}})+ \col(\bullet_{x_{1}}) \cdot (\bullet_{x_{2}}\cdots\bullet_{x_{m}})
& \text{ if }  F=\bullet_{x_{1}}\cdots \bullet_{x_{m}} \text{ with } m\geq 2 \text{ and } x_i\in X.
\end{array}
\right .
\label{eq:dele}
\end{equation}
Here in the third case, the definition of $\col$ reduces to the induction on breadth and the left and right actions are given
in Eq.~(\ref{eq:dota}).

For the induction step of $\dep(F)\geq 1$, we reduce the definition to induction on breadth.
If $\bre(F) = 1$, we
write $F=B_{\omega}^{+}(\overline{F})$ for some $\omega\in \Omega$ and $\overline{F} ,\overline{F}  \in \rfs$, and define
\begin{equation}
\col(F):=\col B_{\omega}^{+}(\overline{F}) := \overline{F} \otimes \etree + (\id\otimes B_{\omega}^{+})\col(\overline{F}).
\label{eq:dbp}
\end{equation}
In other words
\begin{align}
\col B_{\omega}^{+}= \mathrm{id} \otimes \etree + (\id\otimes B_{\omega}^{+})\col.
\label{eq:cdbp}
\end{align}
We call Eq.~(\ref{eq:dbp}) or equivalently Eq.~(\ref{eq:cdbp}) the {\bf infinitesimal 1-cocycle condition}.
If $\bre(F) \geq 2$, we write $F=T_{1}T_{2}\cdots T_{m}$ for some $m\geq 2$ and $T_1, \ldots, T_m \in \rts$, and define
\begin{equation}
\col(F)=T_{1}\cdot \col(T_{2}\cdots T_{m})+\col(T_{1})\cdot (T_{2}\cdots T_{m}).
\label{eq:dele1}
\end{equation}

Let us expose some examples for better insight into $\col$.
\begin{exam}\label{exam:cop}
Let $x,y \in X$ and $\alpha, \beta, \omega\in \Omega$. Then
\begin{align*}
\col(\tdun{$\alpha$})&= \col( B_{\alpha}^{+}(\etree)) =\etree\ot \etree, \\
\col(\tddeux{$\alpha$}{$x$})&= \col( B_{\alpha}^{+}(\tdun{$x$})) =\tdun{$x$}\ot \etree + \etree\ot \tdun{$\alpha$},\\
\col(\tdun{$\beta$}\tddeux{$\alpha$}{$x$})&=\tdun{$\beta$} \cdot \col( \tddeux{$\alpha$}{$x$}) +\col(\tdun{$\beta$}) \cdot \tddeux{$\alpha$}{$x$} = \tdun{$\beta$}\tdun{$x$}\ot \etree+\tdun{$\beta$}\ot \tdun{$\alpha$}+\etree\ot \tddeux{$\alpha$}{$x$},\\
\col(\tdquatretrois{$\omega$}{$\alpha$}{$x$}{$\beta$})&=\col(B_{\omega}^{+}(\tdun{$\beta$}\tddeux{$\alpha$}{$x$}))
=\tdun{$\beta$}\tddeux{$\alpha$}{$x$}\ot \etree+\tdun{$\beta$}\tdun{$x$}\ot \tdun{$\omega$}+\tdun{$\beta$}\ot \tddeux{$\omega$}{$\alpha$}+ \etree\ot \tdtroisdeux{$\omega$}{$\alpha$}{$x$}.
\end{align*}
\end{exam}
\begin{exam}
Foissy~\cite{Foi09} studied  an $\epsilon$-Hopf algebra on (undecorated) planar rooted forests, using a different coproduct $\Delta$ given by
\begin{align*}
\Delta(F) :=
\left\{
\begin{array}{ll}
\etree \ot \etree & \text{ if } F = \etree, \\
F \ot \etree+ (\id \ot B^+)\Delta(\overline{F}) & \text{ if } F = B^+(\overline{F}),\\
F_1 \cdot \Delta(F_2)+ \Delta(F_1) \cdot F_2 - F_1\ot F_2
& \text{ if }  F=F_1F_2.
\end{array}
\right .
\end{align*}
Then
\begin{align*}
\Delta(\tun) =& \tun\ot \etree + \etree \ot \tun,\\
\Delta(\tdeux)=&\tdeux\ot \etree +\tun \ot \tun +\etree \ot \tdeux,\\
\Delta(\tun \tdeux)=&\tun \tdeux \ot \etree+ \tun \tun \ot \tun + \tun\ot \tdeux  + \etree\ot \tun \tdeux, \\
\Delta(\tquatretrois)=&\tquatretrois \ot \etree + \tun \tdeux \ot \tun  +\tun\tun\ot \tdeux + \tun \ot \ttroisdeux +\etree \ot\tquatretrois,
\end{align*}
which are different from the corresponding ones suppressed decorations in Example~\ref{exam:cop}.
\end{exam}

\begin{remark}
The infinitesimal 1-cocycle condition in~Eq.~(\ref{eq:dbp}) is different from the 1-cocycle condition employed in~\cite{CK98} given by
\begin{equation}
\Delta (F)=\Delta B^{+}(\overline{F})= F \otimes \etree + (\id\otimes B^{+})\Delta(\overline{F})\,\text{ for }\, F=B^+(\overline{F})\in \calt.
\label{eq:usuco}
\end{equation}
\end{remark}

\subsection{A combinatorial description of $\col$}
Next we give a combinatorial description of the coproduct $\Delta_{\epsilon}$. Denote by $V(F)$ the set of vertices of a forest $F$.
Foissy erected several order relations on $V(F)$ of a planar rooted forest $F$ and
proposed the concept of biideals of $F$~\cite{Foi02, Foi09},
which can be utilized to our case.

\begin{defn}
Let $F = T_1 \cdots T_n\in \rfs$ with $T_1, \ldots, T_n\in \rts$ and $n\geq 0$,
and let $u,v\in V(F)$.
\begin{enumerate}
\item $u\leqh v$ ({\bf being higher}) if there exists a (directed) path from $u$ to $v$ in $F$ and the edges of $F$ are oriented from roots to leaf vertices.
\item $u\leql v$ ({\bf being more on the left}) if $u$ and $v$ are not comparable for $\leqh $ and one of the following assertions is satisfied:
   \begin{enumerate}
   \item  $u$ is a vertex of $T_i$ and $v$ is a vertex of $T_j$, with $1\leq j< i \leq n$.
   \item $u$ and $v$ are vertices of the same $T_i$, and $u\leql v$ in the forest obtained from $T_i$ by deleting its root.
    \end{enumerate}
\item $u\lhl v$ ({\bf being higher or more on the left}) if $u\leqh v$ or $u\leql v$.
\item A set $I\subseteq V(F)$ is called a {\bf biideal} of $F$, denoted by $I \fid F$, if
$$u\in I\,\text{ and }\, v\in V(F)\,\text{ such that }\, u\lhl v \Rightarrow v\in I.$$
If the biideal $I \neq V(F)$, then $I$ is called a {\bf proper biideal} of $F$, denoted by $I \sfid F$.
\label{it:orderd}
\end{enumerate}
\label{def:order}
\end{defn}

The empty set $\emptyset$ is a biideal of $F$.
Any proper biideal $I$ of $F$ can't contain the root of the right-most tree $T_n$ in $F$,
as the root of the right-most tree $T_n$ is the minimum vertex in $V(F)$ with respect to $\lhl$.

\begin{exam}
Let $F=\tdtroisun{$\alpha$}{$\gamma$}{$\beta$}$. The following arrays give the order relations $\leqh $ and $\leql $ for the vertices of $F$, respectively.
The symbol $\times$ means that the vertices are not comparable for the order.
\begin{table}[H]
\begin{tabular}{|c|c|c|c|}
  \hline
   $\leqh$ & $\bullet_\alpha$& $\bullet_\beta$ & $\bullet_\gamma$ \\
  \hline
   $\bullet_\alpha$ & $=$ & $\leqh$ & $\leqh$\\
  \hline
  $\bullet_\beta$ & $\geq_{{\rm h}}$ & = & $\times$  \\
  \hline
  $\bullet_\gamma$ & $\geq_{{\rm h}}$ & $\times$  & $=$  \\
  \hline
\end{tabular}\quad \quad \quad \quad
\begin{tabular}{|c|c|c|c|}
  \hline
   $\leql$ & $\bullet_\alpha$& $\bullet_\beta$ & $\bullet_\gamma$ \\
  \hline
   $\bullet_\alpha$ & $=$ & $\times$ & $\times$\\
  \hline
  $\bullet_\beta$ & $\times$ & = & $\geq_{{\rm l}}$  \\
  \hline
  $\bullet_\gamma$ & $\times$ & $\leql$  & $=$  \\
  \hline
\end{tabular}
\end{table}
\noindent Then $\bullet_\alpha\lhl \bullet_\gamma\lhl \bullet_\beta$ and  $\tdtroisun{$\alpha$}{$\gamma$}{$\beta$}$ has three proper biideals:
$\emptyset,\, \{\bullet_\beta\} \,\text{ and }\,  \{\bullet_\beta,\bullet_\gamma\}.$
\end{exam}

\begin{remark}
As in~\cite[Sec.~2.1]{Foi09}, for $F\in \rfs$,
\begin{enumerate}
\item the order $\leqh $ is a partial order on $V(F)$, whose Hasse graph is the decorated planar rooted forest $F$;
\item the order $\leql$ is a  partial order on $V(F)$ and the order $\lhl$ is a total order on $V(F)$.
\end{enumerate}
\end{remark}

The proper biideals of $F$ can be characterized precisely as

\begin{lemma}
Let $F\in \rfs$ and let $u_n\lhl \cdots \lhl u_1 $ be its vertices.
Then $F$ has exactly $n$ proper biideals:
\begin{align*}
I_k=\{u_1, \ldots, u_{k-1}\}\, \text{ for }\, k = 1, \ldots, n,
\end{align*}
with the convention that $I_1 = \emptyset$.
\label{lem:comid}
\end{lemma}

\begin{proof}
It is the same as the proof of~\cite[Lemma~13]{Foi09}.
\end{proof}

For any set $I\subseteq V(F)$, let $F_{|I}$ be the derived decorated planar rooted subforest whose vertices are $I$ and edges are the edges of $F$ between elements of $I$. In particular,
\begin{equation}
F_{|\emptyset} := \etree \,\text{ and }\, \etree_{|I}:=0.
\label{eq:efoz}
\end{equation}

\begin{remark}\label{rk:cut}
Let $F=B_{\omega}^{+}(\overline F)$ for some $\omega\in \Omega$ and $\overline{F}\in \rfs$ and $I\subseteq V(\overline F)$. Then
$$B_{\omega}^{+}({\overline F}_{|I})= B_{\omega}^{+}(\overline F)_{| (I\sqcup \{\bullet_\omega\})},$$
illustrated graphically as below:
\begin{align*}
\overline {F}=\begin{tikzpicture}[baseline,scale=0.8]
\draw (0,0) circle (0.5);
\node at (0,0) {$\overline F$};
\end{tikzpicture},\ \
{\overline F}_{|I}=\begin{tikzpicture}[baseline,scale=0.8]
\draw (4.5,-0.5) arc (270:90:0.5)--+(0,-1);
\node at (4.25,0) {$I$};
\end{tikzpicture},\ \ \ \
B_{\omega}^{+}({\overline F}_{|I})=\begin{tikzpicture}[baseline, scale=0.8]
\draw (4.5,-0.5) arc (270:90:0.5)--+(0,-2);
\node at (4.25,0) {$I$};
\fill (4.5,-1.5) circle (2pt) node[right]{$\omega$};
\end{tikzpicture},
B_{\omega}^{+}(\overline F)=\begin{tikzpicture}[baseline,scale=0.8]
\draw (3,0) circle (0.5);
\draw (3,0.5)--+(0,-2);
\node at (2.75,0) {$I$};
\fill (3,-1.5) circle (2pt) node[right]{$\omega$};
\end{tikzpicture},\ \
B_{\omega}^{+}(\overline F)_{| (I\sqcup \{\bullet_\omega\})}=\begin{tikzpicture}[baseline,scale=0.8]
\draw (4.5,-0.5) arc (270:90:0.5)--+(0,-2);
\node at (4.25,0) {$I$};
\fill (4.5,-1.5) circle (2pt) node[right]{$\omega$};
\end{tikzpicture}.
\end{align*}
\end{remark}

\begin{lemma}
Let $F=B_{\omega}^{+}(\overline F)$ for some $\omega\in \Omega$ and $\overline{F}\in \rfs$, and let $I \subseteq V(\overline F)$. Then
$I\fid \overline F$ if and only if $I\fid F$.
\label{lem:bfid}
\end{lemma}

\begin{proof}
It follows from Definition~\ref{def:order}~(\ref{it:orderd}) and the fact that the root $\bullet_\omega$ is minimum in $V(F)$ with respect to $\lhl$.
\end{proof}

\begin{theorem}
With the notations in Lemma~\ref{lem:comid}, we have
\begin{align}
\Delta_{\epsilon}(F) = \sum_{I_k \sfid F} F_{|I_k} \ot F_{|\overline{I}_k} = \sum_{k=1}^n F_{|I_k} \ot F_{|\overline{I}_k},
\label{eq:comb}
\end{align}
where $\overline{I}_k:= V(F)\setminus (I_k \sqcup\{u_k\}) = \{u_{k+1}, \ldots, u_n\}$ for $k=1, \ldots, n$, with the convention that $\overline{I}_n = \emptyset$.
\label{thm:comb}
\end{theorem}

\begin{proof}
The second equality follows from Lemma~\ref{lem:comid}.
It is sufficient to prove the first equality for basis elements $F\in \rfs$ by induction on $n:= |V(F)|$.
If $n=0$, then $F=\etree$ and $F_{|I_k} = \etree_{|I_k} = 0$ by Eq.~(\ref{eq:efoz}). It follows from Eq.~(\ref{eq:dele}) that
$$\col(F)=\col(\etree)=0 = \sum_{I_k \sfid F} F_{|I_k} \ot F_{|\overline{I}_k}.$$

Assume that the result holds for $n<m$ for a $m\geq 1$ and consider the case of $n=m$.
In this case, we use induction on breadth $\bre(F)$.
Since $F\in \rfs$ with $|V(F)| = n\geq 1$, we have $F\neq \etree$ and $\bre(F)\geq 1$.
If $\bre(F) =1$, we have two cases to consider.

\noindent{\bf Case 1.} $F=\bullet_{x}$ for some $x\in X$. In this case, $F$ has only one proper biideal $\emptyset$.
By Eqs.~(\ref{eq:dele}) and~(\ref{eq:efoz}),
$$\col(F)=\col(\bullet_{x})=\etree\ot \etree =  F_{|\emptyset}\ot F_{|\emptyset} = \sum_{I_k \sfid F} F_{|I_k} \ot F_{|\overline{I}_k}.$$

\noindent{\bf Case 2.} $F = B_{\omega}^{+}(\overline{F})$ for some $\overline{F} \in \rfs$ and $\omega\in \Omega$.
By Lemma~\ref{lem:comid}, the proper biideals of $\overline F$ are $I_k$ for $k=1, \ldots, n-1$.
Then
\begin{align*}
\col(F)&= \col\Big(B_{\omega}^{+}(\overline{F})\Big) = \overline{F}\ot \etree+(\id\ot B_{\omega}^{+})\col(\overline{F})  \quad (\text{by Eq.~(\ref{eq:dbp})})\\
&= \overline{F}\ot \etree+(\id\ot B_{\omega}^{+}) \left(\sum_{I_{k}\sfid \overline {F}} \overline{F}_{|I_{k}}\ot \overline{F}_{|\overline{I}_{k}}\right) \quad (\text{by the induction hypothesis})\\
&=F_{|V(\overline{F})}\ot F_{|\emptyset}+\sum_{I_{k}\sfid \overline {F}} \overline{F}_{|I_{k}}\ot B_{\omega}^{+}( \overline{F}_{|\overline{I}_{k}})\\
&=F_{|V(\overline{F})}\ot F_{|\emptyset} +\sum_{I_{k}\sfid \overline {F}} \overline{F}_{|I_{k}}\ot B_{\omega}^{+}\big(\overline{F}_{|V(\overline{F})\setminus (I_{k}\sqcup\{u_{k}\})}\big)\\
&=F_{|V(\overline{F})}\ot F_{|\emptyset} + \sum_{I_{k}\sfid F, I_{k}\neq V(\overline F)} \overline{F}_{|I_{k}}\ot B_{\omega}^{+}\big(\overline{F}_{|V(\overline{F})\setminus (I_{k}\sqcup\{u_{k}\})}\big)\quad(\text{by Lemma~\ref{lem:bfid} })\\
&=F_{|V(\overline{F})}\ot F_{|\emptyset}+\sum_{I_{k}\sfid F, I_{k}\neq V(\overline F)}\overline{F}_{|I_{k}}\ot {B_{\omega}^{+}(\overline{F})}_{|(V(\overline{F})\setminus (I_{k}\sqcup\{u_{k}\}))\sqcup \{\bullet_\omega\}} \quad(\text{by Remark~\ref{rk:cut}}) \\
&=F_{|V(\overline{F})}\ot F_{|\emptyset}+\sum_{I_{k}\sfid F, I_{k}\neq V(\overline F)} {F}_{|I_{k}}\ot F_{|V(F)\setminus (I_{k}\sqcup\{u_{k}\})}\\
&= \sum_{I_k \sfid F} F_{|I_k} \ot F_{|\overline{I}_k}.
\end{align*}

Assume that the result holds for $n=m$ and $\bre(F) <k$, and consider the case of $n=m$ and $\bre(F) =k\geq 2$.
Then we may write $F=TF' $ for some $T \in \rts$ and $F' \in \rfs$.
On the one hand,
\begin{align*}
\col(F)=& \col(T F' )=\ T \cdot \col(F' )+\col(T )\cdot F' \quad (\text{by Eq.~(\ref{eq:dele1})})\\
=& T \cdot \sum_{I_{k_2}\sfid  F' } F'_{|I_{k_2}}\ot F'_{|\overline{I}_{k_2}}+\Big(\sum_{I_{k_1}\sfid  T} T_{|I_{k_1}}\ot T_{|\overline{I}_{k_1}}\Big)\cdot F' \quad (\text{by the induction hypothesis})\\
=& \sum_{I_{k_2}\sfid  F' }T F'_{|I_{k_2}} \ot F'_{|\overline{I}_{k_2}}+\sum_{I_{k_1}\sfid  T} T_{|I_{k_1}}\ot T_{|\overline{I}_{k_1}} F' \quad (\text{by Eq.~(\ref{eq:dota})}).
\end{align*}
On the other hand,
\allowdisplaybreaks{
\begin{align*}
\sum_{I_k \sfid F} F_{|I_k} \ot F_{|\overline{I}_k}&=\sum_{\substack{ I_{k}\sfid  TF',\\ I_k\cap V(F' )\neq \emptyset}} F_{|I_k}\ot F_{|\overline{I}_k}+\sum_{\substack{I_k\sfid  TF',\\ I_k\cap V(F' )=\emptyset}} F_{|I_k}\ot F_{|\overline{I}_k}\\
&=\sum_{\substack{I_k=V(T)\sqcup I_{k_2},\\ \emptyset \neq I_{k_2}\sfid  F'}} F_{|I_k}\ot F_{|\overline{I}_k}+\bigg(\sum_{I_k\sfid  T } F_{|I_k}\ot F_{|\overline{I}_k}+\sum_{I_k=V(T) } F_{|I_k}\ot F_{|\overline{I}_k}\bigg)\\
&=\sum_{\substack{ I_k=V(T)\sqcup I_{k_2},\\ I_{k_2}\sfid  F'}} F_{|I_k}\ot F_{|\overline{I}_k}+\sum_{I_k\sfid  T } F_{|I_k}\ot F_{|\overline{I}_k}\\
&=\sum_{\substack{ I_k=V(T)\sqcup I_{k_2},\\ I_{k_2}\sfid  F'}} F_{|I_k}\ot F_{|V(F)\setminus (I_k \sqcup \{u_k\})}+\sum_{I_k\sfid  T } F_{|I_k}\ot F_{|V(F)\setminus (I_k\sqcup \{u_k\})} \\
&=\sum_{I_{k_2}\sfid  F'} F_{|V(T)\sqcup I_{k_2}}\ot F_{|V(F)\setminus (V(T)\sqcup I_{k_2} \sqcup \{u_{k_2}\})}+\sum_{I_k\sfid  T } F_{|I_k}\ot F_{|(V(T)\setminus (I_k\sqcup \{u_k\}))\sqcup V(F')}\\
&=\sum_{I_{k_2}\sfid  F'} TF'_{|I_{k_2}}\ot F'_{|V(F')\setminus (I_{k_2} \sqcup \{u_{k_2}\})} +\sum_{I_k\sfid  T } F_{|I_k}\ot T_{|V(T)\setminus (I_k\sqcup \{u_k\})}F'\\
&=\sum_{I_{k_2}\sfid  F' }T F'_{|I_{k_2}}\ot F'_{|\overline{I}_{k_2}}+\sum_{I_{k}\sfid  T} T_{|I_{k}}\ot T_{|\overline{I}_{k}} F'.
\end{align*}
}
Hence Eq.~(\ref{eq:comb}) holds. This completes the induction on breadth and so the induction on $|V(F)|$.
\end{proof}

\begin{exam}
\begin{enumerate}
\item \label{exam:a}
Consider $F=\tdun{$\beta$}\tddeux{$\alpha$}{$x$}$. Then $\bullet_\alpha \lhl \bullet_x\lhl \bullet_\beta$.
By Lemma~\ref{lem:comid}, $F$ has three proper biideals $\emptyset$, $\{\bullet_\beta\}$ and $\{\bullet_\beta, \bullet_x\}$. So by Theorem~\ref{thm:comb},
\begin{align*}
\col(\tdun{$\beta$}\tddeux{$\alpha$}{$x$})&=F_{|\emptyset}\ot F_{|V(F)\setminus \{\bullet_\beta\}}
+F_{|\{\bullet_\beta\}}\ot F_{|V(F)\setminus (\{\bullet_\beta\}\sqcup\{\bullet_x\})}
+F_{|\{\bullet_\beta, \bullet_x\}}\ot F_{|V(F)\setminus( \{\bullet_\beta, \bullet_x\}\sqcup\{ \bullet_\alpha\})}\\
&=F_{|\emptyset}\ot F_{| \{\bullet_x, \bullet_\alpha\}}
+F_{|\{\bullet_\beta\}}\ot F_{| \{\bullet_\alpha\}}
+F_{|\{\bullet_\beta, \bullet_x\}}\ot F_{|\emptyset}\\
&=\etree\ot \tddeux{$\alpha$}{$x$}+\tdun{$\beta$}\ot \tdun{$\alpha$}+ \tdun{$\beta$}\tdun{$x$}\ot \etree.
\end{align*}

\item \label{exam:b}
Let $F=\tdquatretrois{$\omega$}{$\alpha$}{$x$}{$\beta$}$. Then $\bullet_\omega \lhl \bullet_\alpha\lhl \bullet_{x}\lhl \bullet_{\beta}$ and so $F$ has four proper biideals $\emptyset$, $\{\bullet_\beta\}$, $\{\bullet_\beta, \bullet_x\}$ and $\{\bullet_\beta, \bullet_x, \bullet_\alpha\}$ by Lemma~\ref{lem:comid}. It follows from Theorem~\ref{thm:comb} that
\begin{align*}
\col(\tdquatretrois{$\omega$}{$\alpha$}{$x$}{$\beta$})=&F_{|\emptyset}\ot F_{|V(F)\setminus \{\bullet_\beta\}}+F_{|\{\bullet_\beta\}}\ot F_{|V(F)\setminus (\{\bullet_\beta\}\sqcup\{\bullet_x\})}+F_{|\{\bullet_\beta, \bullet_x\}}\ot F_{|V(F)\setminus (\{\bullet_\beta, \bullet_x\}\sqcup\{\bullet_\alpha\})}\\
&+F_{|\{\bullet_\beta, \bullet_x,\bullet_\alpha \}}\ot F_{|V(F)\setminus (\{\bullet_\beta, \bullet_x,\bullet_\alpha \}\sqcup\{\bullet_\omega\})}\\
=&F_{|\emptyset}\ot F_{|\{\bullet_x,\bullet_\alpha,\bullet_\omega\}}+F_{|\{\bullet_\beta\}}\ot F_{| \{\bullet_\alpha, \bullet_\omega\}}+F_{|\{\bullet_\beta, \bullet_x\}}\ot F_{|\{\bullet_\omega\}}
+F_{|\{\bullet_\beta, \bullet_x,\bullet_\alpha \}}\ot F_{|\emptyset}\\
=&\etree\ot \tdtroisdeux{$\omega$}{$\alpha$}{$x$}+\tdun{$\beta$}\ot \tddeux{$\omega$}{$\alpha$}+\tdun{$\beta$}\tdun{$x$}\ot \tdun{$\omega$}+ \tdun{$\beta$}\tddeux{$\alpha$}{$x$}\ot \etree.
\end{align*}
\end{enumerate}
Observe that the results in~(\ref{exam:a}) and (\ref{exam:b}) are consistent with the ones in Example~\ref{exam:cop}.
\end{exam}

\subsection{Infinitesimal unitary bialgebras on decorated planar rooted forests}
This subsection is devoted to an $\epsilon$-unitary bialgebra of weight zero on $\hrts$.
We first record some lemmas for a preparation.

\begin{lemma}
Let $F_1, F_2\in \hrts$. Then $\col(F_1 F_2) = F_1 \cdot \col(F_2) + \col(F_1) \cdot F_2$.
\label{lem:colff}
\end{lemma}

\begin{proof}
Similar to the case of $\bre(F)\geq 2$ in the proof of Theorem~\ref{thm:comb}, we obtain
\begin{align*}
 \col(F_1F_2) &= \sum_{I_k \sfid F_1F_2} (F_1F_2)_{|I_k} \ot (F_1F_2)_{|\overline{I}_k} \\
 &= \sum_{\substack{ I_{k}\sfid  F_1 F_2,\\ I_k\cap V(F_2 )\neq \emptyset}} (F_1F_2)_{|I_k}\ot (F_1F_2)_{|\overline{I}_k}+\sum_{\substack{I_k\sfid  F_1F_2,\\ I_k\cap V(F_2) = \emptyset}} (F_1F_2)_{|I_k}\ot (F_1F_2)_{|\overline{I}_k}\\
&= F_1 \cdot \col(F_2) + \col(F_1) \cdot F_2,
\end{align*}
as required.
\end{proof}

\begin{lemma}
The pair $(\hrts, \col)$ is a coalgebra (without counit).
\label{lem:rt1}
\end{lemma}

\begin{proof}
First, it follows from Theorem~\ref{thm:comb} that
$$\col(F)\in \hrts\ot\hrts\,\text{ for }\, F\in \rfs.$$
So we are left to show the coassociative law:
\begin{equation*}
(\id\otimes \col)\col(F)=(\col\otimes \id)\col(F)\,\text{ for } F\in \rfs.
\end{equation*}
Suppose that $F$ has vertices: $u_n\lhl u_{n-1}\lhl \cdots \lhl u_1$.
For simplicity, for $1\leq i<j \leq n$, we denote by
\begin{align*}
[i,j]:= \{u_i, \ldots, u_j\}.
\end{align*}
In particular, for  $1\leq k\leq n$,
$$[1, k-1] = \{u_1, \ldots, u_{k-1} \} = I_k\,\text{ and }\,  [k+1,n]=\{u_{k+1}, \ldots,  u_n\}=\overline{I}_{k}.$$
On the one hand, it follows from Theorem~\ref{thm:comb} that
\begin{align*}
(\col\otimes \id)\col(F)=&\ (\col\otimes \id) \sum_{k=1}^n F_{|I_k} \ot F_{|\overline{I}_k}=\ (\col\otimes \id) \sum_{k=1}^n F_{|[1,k-1]} \ot F_{| [k+1,n]}\\
=&\  \sum_{k=1}^n\col( F_{|[1,k-1])} \ot F_{|[k+1,n]}=\ \sum_{k=1}^n \Big(\sum_{i=1}^{k-1}F_{|[1,i-1]}\ot F_{|[i+1, k-1]}\Big)\ot F_{| [k+1,n]}\\
=&\ \sum_{k=1}^n \sum_{i=1}^{k-1}F_{|[1,i-1]}\ot F_{|[i+1, k-1]}\ot F_{|[k+1,n]}=\ \sum_{k=2}^n \sum_{i=1}^{k-1}F_{|[1,i-1]}\ot F_{|[i+1, k-1]}\ot F_{|[k+1,n]}.
\end{align*}
On the other hand, again by Theorem~\ref{thm:comb},
\begin{align*}
(\id\otimes \col)\col(F)=&\ (\id\otimes \col) \sum_{i=1}^n F_{|I_i} \ot F_{|\overline{I}_i}=\ (\id\otimes \col) \sum_{i=1}^n F_{|[1,i-1]} \ot F_{|[i+1,n]}\\
=&\  \sum_{i=1}^n F_{|[1,i-1]} \ot \col (F_{|[i+1,n]})=\  \sum_{i=1}^n F_{|[1,i-1]} \ot \Big(\sum_{k=i+1}^{n}F_{|[i+1, k-1]}\ot F_{|[k+1, n]} \Big)\\
=&\  \sum_{i=1}^n \sum_{k=i+1}^{n} F_{|[1,i-1]} \ot F_{|[i+1, k-1]}\ot F_{|[k+1, n]} =\  \sum_{k=2}^{n} \sum_{i=1}^{k-1} F_{|[1,i-1]} \ot F_{|[i+1, k-1]}\ot F_{|[k+1, n]} \\
=&\  (\col\otimes \id)\col(F).
\end{align*}
This completes the proof.
\end{proof}

Now we arrive at our main result in this subsection.

\begin{theorem}
The quadruple $(\hrts,\, \conc, \etree, \,\col )$ is an  $\epsilon$-unitary bialgebra of weight zero.
\label{thm:rt2}
\end{theorem}

\begin{proof}
Since $\hrts$ is closed under the concatenation product $\conc$, the $(\hrts,\, \conc, \etree)$ is a unitary algebra.
Then the result follows from Lemmas~\ref{lem:colff} and~\ref{lem:rt1}.
\end{proof}

\subsection{Free $\Omega$-cocycle infinitesimal unitary bialgebras}
The concept of an algebra with (one or more) linear operators was introduced by Kurosh~\cite{Kur60}.
Later Guo~\cite{Guo09} called them operated algebras and constructed free operated algebras on a set.
See also~\cite{BCQ10, Gub}. In this subsection, we conceptualize the mixture of operated algebras and infinitesimal bialgebras,
and characterize the freeness of $\hrts$ in such categories.

\begin{defn}\cite{Guo09} Let $\Omega$ be a nonempty set.
\begin{enumerate}
\item An {\bf $\Omega$-operated monoid } is a monoid $M$ together with a set of operators $P_{\omega}: M\to M$, $\omega\in \Omega$.

\item An {\bf $\Omega$-operated (unitary) algebra } is a (unitary) algebra $A$ together with a set of linear operators $P_{\omega}: A\to A$, $\omega\in \Omega$.
\end{enumerate}
\end{defn}

Inspired by Eq.~(\ref{eq:dbp}), we introduce the following concepts.

\begin{defn}\label{defn:xcobi}
Let $\Omega$ be a nonempty set and $\lambda$ a given element of $\bfk$.
\begin{enumerate}
\item  An  {\bf $\Omega$-operated  $\epsilon$-bialgebra of weight $\lambda$} is an $\epsilon$-bialgebra $H$ of weight $\lambda$ together with a set of linear operators $P_{\omega}: H\to H$, $\omega\in \Omega$. If further $H$ has a unity, $H$ is called an {\bf $\Omega$-operated $\epsilon$-unitary bialgebra of weight $\lambda$}.
\label{it:def1}

\item Let $(H,\, \{P_{\omega}\mid \omega \in \Omega\})$ and $(H',\,\{P'_{\omega}\mid \omega\in \Omega\})$ be two $\Omega$-operated $\epsilon$-bialgebras of weight $\lambda$. A linear map $\phi : H\rightarrow H'$ is called an {\bf $\Omega$-operated $\epsilon$-bialgebra morphism} if $\phi$ is a morphism of $\epsilon$-bialgebras and $\phi \circ P_\omega = P'_\omega \circ\phi$ for $\omega\in \Omega$. The {\bf $\Omega$-operated $\epsilon$-unitary bialgebra morphism} can be defined in the same way.
\label{it:def2}
\end{enumerate}
Since the unit 1 plays a part in Eq.~(\ref{eq:eqiterated}) below, we need to incorporate the unitary property
in the following concepts.
\begin{enumerate}
\setcounter{enumi}{2}
\item \label{it:def3}
An {\bf $\Omega$-cocycle $\epsilon$-unitary bialgebra of weight $\lambda$} is an $\Omega$-operated  $\epsilon$-unitary bialgebra $(H,\,m,\,1, \Delta, \{P_{\omega}\mid \omega \in \Omega\})$ of weight $\lambda$ satisfying the $\epsilon$-cocycle condition:
\begin{equation}
\Delta P_{\omega} = \id \otimes 1+ (\id\otimes P_{\omega}) \Delta \quad \text{ for } \omega\in \Omega.
\label{eq:eqiterated}
\end{equation}

\item The {\bf free $\Omega$-cocycle $\epsilon$-unitary bialgebra of weight $\lambda$ on a set $X$} is an $\Omega$-cocycle $\epsilon$-unitary bialgebra
$(H_{X},\,m_{X}, \,1_X, \Delta_{X}, \,\{P_{\omega}\mid \omega \in \Omega\})$ of weight $\lambda$ together with a set map $j_X: X \to H_{X}$  with the property that,
for any $\Omega$-cocycle $\epsilon$-unitary bialgebra $(H,\,m,\,1, \Delta,\,\{P'_{\omega}\mid \omega \in \Omega\})$ of weight $\lambda$ and any set map
$f: X\to H$ such that $\Delta (f(x))=1\ot 1$ for $x \in X$, there is a unique morphism $\bar{f}:H_X\to H$ of
$\Omega$-operated $\epsilon$-unitary bialgebras such that $\bar{f}\circ j_X =f$.\label{it:def4}
\end{enumerate}
\end{defn}

When $\Omega$ is a singleton set, we will omit the ``$\Omega$".
The following result generalizes the universal properties studied in~\cite{CK98, Fo3, Guo09,Moe01, ZGG16}.
Recall from Eq.~(\ref{eq:rdeff}) that
\begin{align*}
\rfs = \lim_{\longrightarrow} \calf_n=\bigcup_{n=0}^{\infty} \calf_n.
\end{align*}

\begin{lemma}\cite[Theorem~4.5]{ZCGL18}
Let $j_{X}: X\hookrightarrow \rfs$, $x \mapsto \bullet_{x}$ be the natural embedding and $m_{\RT}$ be the concatenation product. Then, we have the following.
\begin{enumerate}
\item
The quadruple $(\rfs, \,\mul,\, 1, \, \{B_{\omega}^+\mid \omega\in \Omega\})$ together with the $j_X$ is the free $\Omega$-operated monoid on $X$.
\label{it:fomonoid}
\item
The quadruple $(\hrts, \,\mul,\,1, \, \{B_{\omega}^+\mid \omega\in \Omega\})$ together with the $j_X$ is the free $\Omega$-operated unitary algebra on $X$.
\label{it:fualg}
\end{enumerate}
\label{lemm:free}
\end{lemma}

\begin{theorem}\label{thm:propm}
Let $X$ be a set and $j_{X}: X\hookrightarrow \rfs$, $x \mapsto \bullet_{x}$ the natural embedding.
Then the quintuple $(\hrts, \,\mul,\,\etree, \, \col,\,\{B_{\omega}^+\mid \omega\in \Omega\})$ together with the $j_X$ is the free $\Omega$-cocycle $\epsilon$-unitary bialgebra of weight zero on $X$.
\label{it:fubialg}
\end{theorem}

\begin{proof}
By Theorem~\ref{thm:rt2}, $(\hrts, \,\mul,\,\etree, \col)$ is an $\epsilon$-unitary bialgebra of weight zero,
and further, together with the $\{B_{\omega}^+\mid \omega\in \Omega\}$, is an $\Omega$-cocycle $\epsilon$-unitary bialgebra of weight zero by Eq.~(\ref{eq:dbp}).

Let $(H,\, m,\,1, \Delta,\, \{P_{\omega}\mid \omega \in \Omega\})$ be an $\Omega$-cocycle $\epsilon$-unitary bialgebra of weight zero
and $f:X\to H$ a set map such that $\Delta(f(x)) = 1\ot 1$ for $x\in X$.
In particular, $(H,\, m,\, 1,\, \{P_{\omega}\mid \omega \in \Omega\})$ is an $\Omega$-operated unitary algebra.
By Lemma~\ref{lemm:free} (\ref{it:fualg}), there exists a unique $\Omega$-operated unitary algebra morphism $\bar{f}:\hrts \to H$ such that $\free{f}\circ j_X={f}$. It remains to check the compatibility of the coproducts $\Delta$ and $\col$ for which we verify
\begin{equation}
\Delta \bar{f} (F)=(\bar{f}\ot \bar{f}) \col (F)\quad \text{for all } F\in \rfs,
\label{eq:copcomp}
\end{equation}
by induction on the depth $\dep(F)\geq 0$.
For the initial step of $\dep(F)=0$,
we have $F = \bullet_{x_1} \cdots \bullet_{x_m}$ for some $m\geq 0$ and $x_1, \cdots, x_m\in X$.
Here we use the convention that $F=\etree$ when $m=0$.
If $m=0$, then by Remark~\ref{remk:units} and Eq.~(\ref{eq:dele}),
\begin{align*}
\Delta  \bar{f} (F) &= \Delta  \bar{f} (\etree)= \Delta (1)=0=(\bar{f}\otimes\bar{f})\col(\etree)=(\bar{f}\otimes\bar{f})\col(F).
\end{align*}
If $m\geq 1$, then
\allowdisplaybreaks{
\begin{align*}
\Delta\bar{f}(\bullet_{x_1} \cdots \bullet_{x_m})=&\ \Delta \Big( \bar{f}(\bullet_{x_1})\cdots\bar{f}(\bullet_{x_{m}}) \Big)\\
=&\ \sum_{i=1}^m \Big(\bar{f}(\bullet_{x_{1}})\cdots\bar{f}(\bullet_{x_{i-1}})\Big)\cdot
\Delta\big(\bar{f}(\bullet_{x_{i}})\big) \cdot \Big(\bar{f}(\bullet_{x_{i+1}})\cdots\bar{f}(\bullet_{x_{m}})\Big)
\quad \text{(by Eq.~(\ref{eq:wib}))}\\
=&\ \sum_{i=1}^m \Big(\bar{f}(\bullet_{x_{1}})\cdots\bar{f}(\bullet_{x_{i-1}})\Big)\cdot
\Delta\big(f(x_i)\big) \cdot \Big(\bar{f}(\bullet_{x_{i+1}})\cdots\bar{f}(\bullet_{x_{m}})\Big)\\
=&\ \sum_{i=1}^m \Big(\bar{f}(\bullet_{x_{1}})\cdots\bar{f}(\bullet_{x_{i-1}})\Big)\cdot
(1\ot 1)\cdot \Big(\bar{f}(\bullet_{x_{i+1}})\cdots\bar{f}(\bullet_{x_{m}})\Big)\\
=&\ \sum_{i=1}^m\bar{f}(\bullet_{x_{1}})\cdots\bar{f}(\bullet_{x_{i-1}})\ot\bar{f}
(\bullet_{x_{i+1}}) \cdots\bar{f}(\bullet_{x_{m}})\\
=&\ \sum_{i=1}^m\bar{f}(\bullet_{x_{1}}\cdots\bullet_{x_{i-1}})\ot\bar{f}(\bullet_{x_{i+1}} \cdots\bullet_{x_{m}})\\
=&\ \sum_{i=1}^m (\bar{f}\ot\bar{f})(\bullet_{x_{1}}\cdots \bullet_{x_{i-1}}\ot \bullet_{x_{i+1}} \cdots \bullet_{x_{m}})\\
=&\ (\bar{f}\ot\bar{f})\left( \sum_{i=1}^m \bullet_{x_{1}}\cdots \bullet_{x_{i-1}}\ot \bullet_{x_{i+1}} \cdots \bullet_{x_{m}}\right)\\
=&\ (\bar{f}\ot\bar{f})\col(\bullet_{x_1} \cdots \bullet_{x_m})\quad(\text{by Eq.~(\ref{eq:comb})}).
\end{align*}
}
Suppose that Eq.~(\ref{eq:copcomp}) holds for $\dep(F)\leq n$ for an $n\geq 0$ and consider the case of $\dep(F)=n+1$.
For this case we apply the induction on the breadth $\bre(F)$. Since $\dep(F)=n+1\geq1$, we have $F\neq \etree$ and $\bre(F)\geq 1$.
If $\bre(F)= 1 $, we have $F=B_{\omega}^+(\overline{F})$ for some $\overline{F}\in\rfs$ and $\omega\in \Omega$.  Then
\allowdisplaybreaks{
\begin{align*}
\Delta \bar{f}(F)&=\Delta \bar{f} (B_{\omega}^{+}(\overline{F}))= \Delta P_\omega(\bar{f} (\overline{F}))= \bar{f}(\overline{F}) \ot 1+ (\id\ot P_\omega)\Delta(\bar{f} (\overline{F})) \quad(\text{by Eq.~(\ref{eq:eqiterated})})\\
&=\bar{f}(\overline{F})\ot 1+ (\id\ot P_\omega)(\bar{f}\ot \bar{f}) \col (\overline{F}) \quad(\text{by the induction hypothesis on~}\dep(F)) \\
&=\bar{f}(\overline{F})\ot 1+ (\bar{f}\ot P_\omega\bar{f}) \col (\overline{F})\\
&=\bar{f}(\overline{F})\ot 1+ (\bar{f}\ot \bar{f}B_{\omega}^+) \col (\overline{F})
\quad(\text{by $\bar{f}$ being an operated algebra morphism}) \\
&=(\bar{f}\ot \free{f})\Big(\overline{F}\ot \etree +(\id\ot B_{\omega}^+)\col (\overline{F})\Big) \\
&=(\bar{f}\ot \bar{f}) \col (B_{\omega}^+(\overline{F}))=(\bar{f}\ot \bar{f}) \col (F).
\end{align*}
}
Assume that Eq.~(\ref{eq:copcomp}) holds for $\dep(F)=n+1$ and $\bre(F)\leq m$, in addition to $\dep(F)\leq n$ by the first induction hypothesis, and consider the case when $\dep(F)=n+1$ and $\bre(F)=m+1\geq 2$. Then we can write $F=F_{1}F_{2}$ for some $F_{1},F_{2}\in\rfs$ with $0< \bre(F_{1}), \bre(F_{2}) < m+1$.
Using the Sweedler notation, we can write
\begin{align*}
\Delta_{\epsilon}(F_1)=\sum_{(F_1)}F_{1(1)}\otimes F_{1(2)} \text{\ and \ } \Delta_{\epsilon}(F_2)=\sum_{(F_2)}F_{2(1)}\otimes F_{2(2)}.
\end{align*}
By the induction hypothesis on the breadth,
\begin{equation}
\begin{aligned}
&\Delta(\bar{f}(F_{1}))=(\bar{f} \ot \bar{f})\col(F_{1})=\sum_{(F_{1})}\bar{f}(F_{1(1)})\ot\bar{f} (F_{1(2)}),\\
&\Delta(\bar{f}(F_{2}))=(\bar{f} \ot \bar{f})\col(F_{2})=
\sum_{(F_{2})}\bar{f}(F_{2(1)})\ot\bar{f} (F_{2(2)}).
\end{aligned}
\label{eq:df12}
\end{equation}
Consequently,
\allowdisplaybreaks{
\begin{align*}
\Delta \bar{f}(F)=&\ \Delta \bar{f} (F_{1}F_{2})=\Delta(\bar{f}(F_1)\bar{f}(F_2))= \bar{f}(F_{1})\cdot \Delta (\bar{f}(F_{2}))+\Delta(\bar{f}(F_{1})) \cdot \bar{f}(F_{2})\\
=&\ \bar{f}(F_{1})\cdot \bigg(\sum_{(F_{2})}\bar{f}(F_{2(1)})\ot \bar{f}(F_{2(2)})\bigg)+\bigg(\sum_{(F_{1})}\bar{f}(F_{1(1)})\ot\bar{f} (F_{1(2)})\bigg) \cdot \bar{f}(F_{2}) \quad (\text{by Eq.~(\ref{eq:df12})})\\
=&\ \sum_{(F_{2})}\bar{f}(F_{1})\bar{f}(F_{2(1)})\ot\bar{f}(F_{2(2)})+\sum_{(F_{1})}\bar{f}(F_{1(1)})\ot\bar{f} (F_{1(2)})\bar{f}(F_{2}) \quad (\text{by Eq.~(\ref{eq:dota})})\\
=&\ \sum_{(F_{2})}\bar{f}(F_{1}F_{2(1)})\ot\bar{f}(F_{2(2)})+\sum_{(F_{1})}\bar{f}(F_{1(1)})\ot\bar{f} (F_{1(2)}F_{2})\\
=&\ (\bar{f}\ot \bar{f})\left(\sum_{(F_{2})}F_{1}F_{2(1)}\ot F_{2(2)}+\sum_{(F_{1})}F_{1(1)}\ot F_{1(2)}F_{2} \right)\\
=&\ (\bar{f}\ot \bar{f})\left(F_{1}\cdot \sum_{(F_{2})}F_{2(1)}\ot F_{2(1)}+ \Big(\sum_{(F_{1})}F_{1(1)}\ot F_{1(2)}\Big) \cdot F_{2} \right) \quad (\text{by Eq.~(\ref{eq:dota})})\\
=&\ (\bar{f} \ot \bar{f})(F_{1} \cdot \col(F_{2})+\col(F_{1})\cdot F_{2})= (\bar{f}\ot \bar{f})\col(F_{1}F_{2})  \quad(\text{by Lemma~\ref{lem:colff}})\\
=&\ (\bar{f}\ot \bar{f})\col(F).
\end{align*}
}
This completes the induction on breadth and hence the induction on depth.
\end{proof}

\begin{remark}
Guo~\cite{Guo09} constructed the free $\Omega$-operated monoid on a set $X$ in terms of
Motzkin paths $\mathcal{P}(X, \Omega)$ and Motzkin words $\mathcal{W}(X, \Omega)$, respectively.
By the uniqueness of free objects, we have
$$\rfs\cong \mathcal{P}(X, \Omega) \cong \mathcal{W}(X, \Omega),$$
as $\Omega$-operated monoids.
\end{remark}

\section{Pre-Lie algebras and weighted infinitesimal bialgebras}\label{sec:preLie}
In this section, we derive a pre-Lie algebra from an $\epsilon$-bialgebra of weight $\lambda$.
We refer to~\cite{Aguu02} for the classical result in the case of weight zero.

\subsection{Pre-Lie algebras from weighted infinitesimal bialgebras}
Pre-Lie algebras are a class of nonassociative algebras and have many other names such as left-symmetric algebras,
Koszul-Vinberg algebras, quasi-associative algebras and so on.
See the surveys~\cite{Bai, Man11} for more details.

\begin{defn}
A {\bf (left) pre-Lie algebra} is a $\bfk $-module $A$ together with a binary operation $\rhd: A\ot A \rightarrow A$ satisfying:
\begin{align}
(a\rhd b)\rhd c-a\rhd (b\rhd c)=(b\rhd a)\rhd c-b\rhd (a\rhd c)\, \text{ for }\, a,b,c\in A.
\label{eq:preli0}
\end{align}
\end{defn}

The close relation between pre-Lie algebras and Lie algebras is characterized by the following result.

\begin{lemma}\cite[Theorem~1]{Ger63}
Let $(A, \rhd)$ be a pre-Lie algebra. Define for elements in $A$ a new multiplication by setting
\begin{align*}
[a, b] :=a\rhd b-b\rhd a\,\text{ for }\, a,b\in A.
\end{align*}
Then $(A, [_{-}, _{-}])$ is a Lie algebra.
\label{lem:preL}
\end{lemma}

The following result captures the connection from weighted $\epsilon$-bialgebras to pre-Lie algebras.

\begin{theorem}
Let $(A, m, \Delta)$ be an $\epsilon$-bialgebra of weight $\lambda$.
Then $(A, \rhd)$ is a pre-Lie algebra, where
\begin{align}
\rhd: A\ot A \to A, \, a\ot b \mapsto a\rhd b:=\sum_{(b)}b_{(1)}a b_{(2)}.
\label{eq:preope}
\end{align}
\label{thm:preL}
\end{theorem}

\begin{proof}
Let $a,b, c\in A$. It follows from Eq.~(\ref{eq:wib}) that
\begin{align*}
\Delta(\sum_{(c)}c_{(1)}bc_{(2)})=&\sum_{(c)}\Delta(c_{(1)}bc_{(2)})=\sum_{(c)}\bigg(c_{(1)}b\cdot \Delta(c_{(2)})+\Delta(c_{(1)}b)\cdot c_{(2)}+\lambda c_{(1)}b \ot c_{(2)}\bigg)\\
=&\sum_{(c)}\bigg(c_{(1)}b\cdot \Delta(c_{(2)})+\Big(c_{(1)}\cdot \Delta(b)+\Delta(c_{(1)})\cdot b+\lambda(c_{(1)}\ot b)\Big)\cdot c_{(2)}+\lambda c_{(1)}b \ot c_{(2)}\bigg)\\
=&\sum_{(c)}\bigg(c_{(1)}b\cdot \Delta(c_{(2)})+c_{(1)}\cdot \Delta(b)\cdot c_{(2)}+\Delta(c_{(1)})\cdot b c_{(2)}+\lambda c_{(1)}\ot b c_{(2)}+\lambda c_{(1)}b \ot c_{(2)}\bigg)\\
=&\sum_{(c)}\bigg( \sum_{(c_{(2)})}c_{(1)}bc_{(2)(1)}\ot c_{(2)(2)} +  \sum_{(b)}c_{(1)}b_{(1)}\ot b_{(2)}c_{(2)}+ \sum_{(c_{(1)})}c_{(1)(1)}\ot c_{(1)(2)}bc_{(2)}\\
&+\lambda c_{(1)}\ot bc_{(2)}+\lambda c_{(1)}b\ot c_{(2)}\bigg)\\
=& \sum_{(c)}c_{(1)}bc_{(2)}\ot c_{(3)}+\sum_{(c),(b)}c_{(1)}b_{(1)}\ot b_{(2)}c_{(2)}+ \sum_{(c)}c_{(1)}\ot c_{(2)}bc_{(3)} \\
&+\lambda \sum_{(c)} c_{(1)}\ot bc_{(2)}+ \lambda \sum_{(c)}c_{(1)}b\ot c_{(2)}\quad \text{(by the coassociative law)}.
\end{align*}
Together this with Eq.~(\ref{eq:preope}), we obtain
\begin{align*}
a\rhd (b\rhd c)=& a\rhd (\sum_{(c)}c_{(1)}b c_{(2)})\\
=& \sum_{(c)}c_{(1)}bc_{(2)}a c_{(3)} +\sum_{(c),(b)}c_{(1)}b_{(1)}a b_{(2)}c_{(2)}+ \sum_{(c)}c_{(1)}a c_{(2)}bc_{(3)}\\
&+ \lambda \sum_{(c)}c_{(1)}a bc_{(2)}+\lambda \sum_{(c)}c_{(1)}ba c_{(2)}.
\end{align*}
Moreover,
\begin{align*}
(a\rhd b)\rhd c=(\sum_{(b)}b_{(1)}a b_{(2)})\rhd c=\sum_{(c),(b)}c_{(1)}b_{(1)}a b_{(2)}c_{(2)}.
\end{align*}
Thus
\begin{align}\label{eq:preid}
(a\rhd b)\rhd c-a\rhd (b\rhd c)=&- \sum_{(c)}c_{(1)}bc_{(2)}a c_{(3)}-\sum_{(c)}c_{(1)}a c_{(2)}bc_{(3)}
-\lambda\sum_{(c)} \bigg(c_{(1)}a bc_{(2)}+c_{(1)}ba c_{(2)}\bigg).
\end{align}
Observe that Eq.~(\ref{eq:preid}) is symmetric in $a$ and $b$. Hence
\begin{align*}
(a\rhd b)\rhd c-a\rhd (b\rhd c)=(b\rhd a)\rhd c-b\rhd (a\rhd c),
\end{align*}
and so $(A, \rhd)$ is a pre-Lie algebra.
\end{proof}

\subsection{Pre-Lie algebras from weighted commutative infinitesimal bialgebras}
\begin{defn}
Let $(A, \rhd)$ be a pre-Lie algebra. Then the pair $(A, \rhd)$ is  called {\bf a Novikov algebra} if satisfies
$$(a\rhd b) \rhd c = (a\rhd c) \rhd b \, \text{ for all }\, a, b \in A. $$
\end{defn}
We now give a new way to construct a pre-Lie algebra from a weighted infinitesimal bilagebra.

\begin{lemma}\label{lem:cd}
Let $(A, \mu, \Delta)$ be an $\epsilon$-bialgebra of weight $\lambda$. Denote
$\mathcal{D}=\mu \Delta: A\rightarrow A$. Then
\begin{align*}
\mathcal{D}(ab)=a\mathcal{D}(b)+\mathcal{D}(a)b+\lambda ab.
\end{align*}
\end{lemma}
\begin{proof}
By Eq.~(\ref{eq:wib}), we have
\begin{align*}
\mathcal{D}(ab)=\mu \Delta(ab)=&\ \mu \Big(a\cdot \Delta(b)+\Delta(a)\cdot b+\lambda a\ot b\Big)\\
=&\ \mu \left(\sum_{(b)}ab_{(1)}\ot b_{(2)}+\sum_{(a)}a_{(1)}\ot a_{(2)}b+\lambda a\ot b \right)\\
=&\ \sum_{(b)}ab_{(1)} b_{(2)}+ \sum_{(a)}a_{(1)} a_{(2)}b+\lambda a b\\
=&\ a\mathcal{D}(b)+\mathcal{D}(a)b+\lambda ab,
\end{align*}
as required.
\end{proof}

This motivates us to propose the following concept, which generalizes the classical derivation.

\begin{defn}\label{def:cderi}
Let $\lambda$ be a given element of $\bfk$, and let $A$ be an algebra.
Then a linear map $\mathcal{D}:A\rightarrow A$ is called a {\bf modified derivation of weight $\lambda$} if
\begin{align}
\mathcal{D}(ab) =a \mathcal{D}(b)+\mathcal{D}(a)b+\lambda ab \, \text{ for }\, a, b \in A.
\label{eq:cderi}
\end{align}
\end{defn}

\begin{remark}
\begin{enumerate}
\item The classical derivation is a modified derivation of weight zero, which is significant in differential algebras and Lie algebras.

\item The operator $\mathcal{D}$ in Eq.~(\ref{eq:cderi}) is a special differential type operator~\cite[Conjecture~4.7~(1)]{GSZ13}.
\end{enumerate}
\end{remark}

More generally, we have
\begin{theorem}\label{thm:wGD2}
Let $\kappa, \lambda$ be given elements of $\bfk$ and $(A, \mu )$ a commutative algebra and $\mathcal{D}: A\rightarrow A$ a modified derivation of weight $\lambda$. Define
\begin{align}
a \rhd_{\mathcal{D}} b:= \mu (a\ot \mathcal{D}(b))+\kappa ab =a \mathcal{D}(b)+\kappa ab. \label{eq:newpre2}
\end{align}
Then $(A, \rhd_{\mathcal{D}})$ is a pre-Lie algebra and a  Novikov algebra.
\end{theorem}

\begin{proof}
For $a, b, c\in A$,
\begin{align*}
&\ (a\rhd_{\mathcal{D}} b)\rhd_{\mathcal{D}} c-a\rhd_{\mathcal{D}} (b\rhd_{\mathcal{D}} c)\\
=&\ (a\mathcal{D} (b)+\kappa ab)\rhd_{\mathcal{D}} c-a\rhd_{\mathcal{D}}(b\mathcal{D}(c)+\kappa bc)\quad (\text{by Eq.~(\ref{eq:newpre2})})\\
=&\ a\mathcal{D}(b)\mathcal{D}(c)+\kappa a\mathcal{D}(b)c+\kappa ab \mathcal{D}(c)+\kappa^2abc-a\mathcal{D}(b\mathcal{D}(c))\\
&\ -\kappa ab\mathcal{D}(c)-\kappa a \mathcal{D}(bc)-\kappa^2abc\\
=&\ a\mathcal{D}(b)\mathcal{D}(c)-a\mathcal{D}(b\mathcal{D}(c))+\kappa a\mathcal{D}(b)c-\kappa a \mathcal{D}(bc)\\
=&\ a\mathcal{D}(b)\mathcal{D}(c)-a\Big(b\mathcal{D}^2(c)  +\mathcal{D}(b)\mathcal{D}(c)+\lambda b\mathcal{D}(c) \Big)+ \kappa a\mathcal{D}(b)c \\
&\ - \kappa a \Big(b\mathcal{D}(c) +\mathcal{D}(b)c+\lambda bc \Big)\quad (\text{by Eq.~(\ref{eq:cderi})})\\
=&\ -ab\mathcal{D}^2(c) -\lambda ab\mathcal{D}(c)
- \kappa a b\mathcal{D}(c) -\kappa\lambda abc \\
=&\ -ab \Big(\mathcal{D}^2(c) +(\kappa+\lambda) \mathcal{D}(c)+\kappa\lambda c\Big).
\end{align*}
Then $(A, \rhd_{\mathcal{D}})$ is a pre-Lie algebra follows from the fact that the above identity is symmetric in $a$ and $b$ when $A$ is a commutative algebra. Next we show that $(A, \rhd_{\mathcal{D}})$ is further a Novikov algebra. For $a, b, c\in A$, we have
\begin{align*}
(a\rhd_{\mathcal{D}} b)\rhd_{\mathcal{D}} c=&\ \Big(a\mathcal{D}(b)+\kappa ab\Big)\mathcal{D}(c)+\kappa \Big(a\mathcal{D}(b)+\kappa ab\Big)c\\
=&\ a\mathcal{D}(b)\mathcal{D}(c)+\kappa ab\mathcal{D}(c)+\kappa a\mathcal{D}(b)c+\kappa^2 abc\\
=&\ a\mathcal{D}(c)\mathcal{D}(b)+\kappa a\mathcal{D}(c)b+\kappa a c\mathcal{D}(b)+\kappa^2 abc\quad(\text{by $A$ is commutative})\\
=&\ a\mathcal{D}(c)\mathcal{D}(b)+\kappa a c\mathcal{D}(b)+\kappa a\mathcal{D}(c)b+\kappa^2 abc\\
=&\ \Big(a\mathcal{D}(c)+\kappa ac\Big)\mathcal{D}(b)+\kappa \Big(a\mathcal{D}(c)+\kappa ac\Big)b\\
=&\ (a\rhd_{\mathcal{D}} c)\rhd_{\mathcal{D}} b.
\end{align*}
This completes the proof.
\end{proof}

As a consequence, we obtain a weighted version of the Gel'fand-Dorfman theorem on Novikov algebra.

\begin{coro}[{\bf Weighted Gel'fand-Dorfman}]\label{thm:wGD1}
Let $(A, \mu )$ be a commutative algebra and $\mathcal{D}: A\rightarrow A$ a modified derivation of weight $\lambda$. Define
\begin{align}
a \rhd_{\mathcal{D}} b:= \mu (a\ot \mathcal{D}(b)) =a \mathcal{D}(b). \label{eq:newpre1}
\end{align}
Then $(A, \rhd_{\mathcal{D}})$ is a pre-Lie algebra and a Novikov algebra.
\end{coro}

\begin{proof}
It follows from Theorem~\ref{thm:wGD2} by taking $\kappa=0$.
\end{proof}

\begin{coro}\label{coro:newpre}
Let $\kappa, \lambda$ be given elements of $\bfk$ and  $(A, \mu, \Delta)$ a commutative $\epsilon$-bialgebra of weight $\lambda$. Define
\begin{align*}
\rhd_\epsilon: A\ot A \to A, \, a\ot b \mapsto a\rhd_\epsilon b:=\sum_{(b)}a b_{(1)} b_{(2)}+\kappa ab.
\end{align*}
Then $(A, \rhd_\epsilon)$ is a pre-Lie algebra and a Novikov algebra.
\end{coro}

\begin{proof}
It follows from Lemma~\ref{lem:cd} and Theorem~\ref{thm:wGD2}.
\end{proof}

\begin{remark}
We would like to emphasize that Theorem~\ref{thm:wGD2} is very general. More precisely,
\begin{enumerate}
\item the Gel'fand-Dorfman theorem on Novikov algebra is the case of $\kappa=\lambda=0$ in Theorem~\ref{thm:wGD1};
\item a Novikov algebra introduced and studied by Filipov~\cite{Fil89} is the case of $\lambda=0$ in Theorem~\ref{thm:wGD1};
\item for a fixed element $\kappa \in A$, Eq.~(\ref{eq:newpre2}) also defines a Novikov algebra, which generalize the result studied by Xu~\cite{Xu96}, see also \cite[Remark~3.7]{Bai}.
\end{enumerate}
\end{remark}

\subsection{New pre-Lie algebras on decorated planar rooted forests}
In this subsection, as an application of Theorem~\ref{thm:preL}, we equip  $\hrts$ with a pre-Lie algebraic structure $(\hrts, \rhd_{\RT})$
and a Lie algebraic structure $(\hrts, [_{-}, _{-}]_{\RT})$.  We also give combinatorial descriptions of $\rhd_{\RT}$ and $[_{-}, _{-}]_{\RT}$.

\begin{theorem}
Let $\hrts$ be the $\epsilon$-unitary bialgebra of weight zero in Theorem~\ref{thm:rt2}.
\begin{enumerate}
\item The pair $(\hrts, \rhd_{\RT})$ is a pre-Lie algebra, where
\begin{equation*}
F_1\rhd_{\RT} F_2: =\sum_{(F_2)}F_{2(1)}F_1 F_{2(2)} \,\text{ for } \, F_1, F_2\in \hrts.
\end{equation*}

\item The pair $(\hrts, [_{-}, _{-}]_{\RT})$ is a Lie algebra, where
\begin{align*}
[F_1, F_2]_{\RT}:=F_1\rhd_{\RT} F_2-F_2\rhd_{\RT} F_1 \,\text{ for } \, F_1, F_2\in \hrts.
\end{align*}
\end{enumerate}
\label{thm:preope}
\end{theorem}

\begin{proof}
By Theorems~\ref{thm:rt2} and~\ref{thm:preL}, $(\hrts, \rhd_{\RT})$ is a pre-Lie algebra.
The remainder follows from Lemma~\ref{lem:preL}.
\end{proof}

Combinatorial descriptions of $\rhd_{\RT}$ and $[_{-}, _{-}]_{\RT}$ on $\hrts$ can also be given.
With the notations in Lemma~\ref{lem:comid} and Theorem~\ref{thm:comb}, we have

\begin{coro}\label{coro:preLcomb}
For any $F_1, F_2\in \hrts$,
\begin{align}
F_1\rhd_{\RT} F_2=\sum_{I_{k_2}\sfid  F_2} F_{2|I_{k_2}}F_1F_{2|\overline{I}_{k_2}},
\label{eq:precom}
\end{align}
and
\begin{align}
[F_1, F_2]_{\RT}=\sum_{I_{k_2}\sfid  F_2} F_{2|I_{k_2}}F_1 F_{2|\overline{I}_{k_2}} -\sum_{I_{k_1}\sfid  F_1} F_{1|I_{k_1}}F_2 F_{1|\overline{I}_{k_1}}.
\label{eq:liecom}
\end{align}
\end{coro}
\begin{proof}
It follows directly from Theorems~\ref{thm:comb} and~\ref{thm:preope}.
\end{proof}

\begin{exam}\label{exam:preL}
Let $F_1=\tdun{$x$}$,  $F_2=\tddeux{$\alpha$}{$\beta$}$,  $F_3=\tdun{$\gamma$}\tddeux{$\omega$}{$y$}$
with $\alpha, \beta,  \gamma , \omega\in \Omega$ and $x, y\in X$.
By Lemma~\ref{lem:comid}, $F_1$ has one proper biideal $\emptyset$, $F_2$ has two proper biideals $\emptyset$ and $\{\bullet_\beta\}$ and
$F_3$ has three proper biideals $\emptyset$, $\{\bullet_\gamma\}$ and $\{\bullet_\gamma, \bullet_y\}$.
It follows from Eqs.~(\ref{eq:precom}) and~(\ref{eq:liecom}) that
\begin{align*}
F_1\rhd_{\RT} F_2&=F_{2|\emptyset}F_1 F_{2|\{\bullet_\alpha\}}+F_{2|\{\bullet_\beta\}}F_1 F_{2|\emptyset}
=\etree F_1 \,\tdun{$\alpha$} + \tdun{$\beta$} \, F_1 \etree
=\tdun{$x$}\tdun{$\alpha$}+ \tdun{$\beta$} \tdun{$x$},\\
F_2\rhd_{\RT} F_1&=\etree F_2\etree =\tddeux{$\alpha$}{$\beta$},\, \ \ \ \ \ \
F_2\rhd_{\RT} F_3=\tddeux{$\alpha$}{$\beta$} \tddeux{$\omega$}{$y$} +\tdun{$\gamma$} \tddeux{$\alpha$}{$\beta$} \tdun{$\omega$}+\tdun{$\gamma$}\tdun{$y$}\tddeux{$\alpha$}{$\beta$},\\
[F_1, F_2]_{\RT}&=\tdun{$x$}\tdun{$\alpha$}+\tdun{$\beta$}\tdun{$x$}-\tddeux{$\alpha$}{$\beta$}.
\end{align*}
Next, we check that $F_1$, $F_2$ and $F_3$ satisfy Eq.~(\ref{eq:preli0}).
By Theorem~\ref{thm:comb},
\begin{align*}
\col(F_3)=&~\col(\tdun{$\gamma$}\tddeux{$\omega$}{$y$})=\etree \ot \tddeux{$\omega$}{$y$}+\tdun{$\gamma$}\ot \tdun{$\omega$}+\tdun{$\gamma$}\tdun{$y$}\ot \etree,\\
\col(F_2\rhd_{\RT} F_3)=&~\etree\ot \tdun{$\alpha$}\tddeux{$\omega$}{$y$}
+\tdun{$\beta$}\ot \tddeux{$\omega$}{$y$}+\tddeux{$\alpha$}{$\beta$}\ot \tdun{$\omega$}
+\tddeux{$\alpha$}{$\beta$}\tdun{$y$}\ot \etree+\etree \ot\tddeux{$\alpha$}{$\beta$}\tdun{$\omega$}
+\tdun{$\gamma$}\ot \tdun{$\alpha$}\tdun{$\omega$}\\
&+\tdun{$\gamma$}\tdun{$\beta$}\ot \tdun{$\omega$} + \tdun{$\gamma$}\tddeux{$\alpha$}{$\beta$}\ot \etree
+\etree \ot \tdun{$y$}\tddeux{$\alpha$}{$\beta$}+\tdun{$\gamma$}\ot \tddeux{$\alpha$}{$\beta$}
+\tdun{$\gamma$}\tdun{$y$}\ot \tdun{$\alpha$}+\tdun{$\gamma$}\tdun{$y$}\tdun{$\beta$}\ot \etree.
\end{align*}
Applying Theorem~\ref{thm:preope},
\begin{align*}
(F_1\rhd_{\RT} F_2)\rhd_{\RT} F_3=&\tdun{$x$}\tdun{$\alpha$}\tddeux{$\omega$}{$y$}
+\tdun{$\beta$}\tdun{$x$}\tddeux{$\omega$}{$y$}+\tdun{$\gamma$}\tdun{$x$}\tdun{$\alpha$}\tdun{$\omega$}
+\tdun{$\gamma$}\tdun{$\beta$}\tdun{$x$}\tdun{$\omega$}+
\tdun{$\gamma$}\tdun{$y$}\tdun{$x$}\tdun{$\alpha$}+\tdun{$\gamma$}\tdun{$y$}\tdun{$\beta$}\tdun{$x$},\\
F_1\rhd_{\RT} (F_2\rhd_{\RT} F_3)=&\tdun{$x$}\tdun{$\alpha$}\tddeux{$\omega$}{$y$}
+\tdun{$\beta$}\tdun{$x$}\tddeux{$\omega$}{$y$}+\tddeux{$\alpha$}{$\beta$}\tdun{$x$}\tdun{$\omega$}
+\tddeux{$\alpha$}{$\beta$}\tdun{$y$}\tdun{$x$}+\tdun{$x$}\tddeux{$\alpha$}{$\beta$}\tdun{$\omega$}
+\tdun{$\gamma$}\tdun{$x$}\tdun{$\alpha$}\tdun{$\omega$}
+\tdun{$\gamma$}\tdun{$\beta$}\tdun{$x$}\tdun{$\omega$}\\
&+\tdun{$\gamma$}\tddeux{$\alpha$}{$\beta$}\tdun{$x$}
+\tdun{$x$}\tdun{$y$}\tddeux{$\alpha$}{$\beta$}+\tdun{$\gamma$}\tdun{$x$}\tddeux{$\alpha$}{$\beta$}
+\tdun{$\gamma$}\tdun{$y$}\tdun{$x$}\tdun{$\alpha$}+\tdun{$\gamma$}\tdun{$y$}\tdun{$\beta$}\tdun{$x$}.
\end{align*}
Thus
\begin{align*}
F_1\rhd_{\RT} (F_2\rhd_{\RT} F_3)-(F_1\rhd_{\RT} F_2)\rhd_{\RT} F_3
=\ &(\tddeux{$\alpha$}{$\beta$}\tdun{$x$}\tdun{$\omega$}+\tdun{$x$}\tddeux{$\alpha$}{$\beta$}\tdun{$\omega$})
+(\tddeux{$\alpha$}{$\beta$}\tdun{$y$}\tdun{$x$}+\tdun{$x$}\tdun{$y$}\tddeux{$\alpha$}{$\beta$})\\
&\ +(\tdun{$\gamma$}\tddeux{$\alpha$}{$\beta$}\tdun{$x$}+\tdun{$\gamma$}\tdun{$x$}\tddeux{$\alpha$}{$\beta$}),
\end{align*}
which is symmetric in $F_1=\tdun{$x$}$ and $F_2=\tddeux{$\alpha$}{$\beta$}$ and hence
\begin{align*}
(F_1\rhd_{\RT} F_2)\rhd_{\RT} F_3-F_1\rhd_{\RT} (F_2\rhd_{\RT} F_3)=(F_2\rhd_{\RT} F_1)\rhd_{\RT} F_3-F_2\rhd_{\RT} (F_1\rhd_{\RT} F_3).
\end{align*}
\end{exam}
\begin{remark}

Let us emphasize that our pre-Lie algebra $(\hrts, \rhd_{\RT})$ is different from the one
introduced by Chapoton and Livernet~\cite{CL01} from the viewpoint of operad.
Under the framework of~\cite{CL01}, Matt gave an exmple~\cite{Mat10}:
\begin{align*}
\tdun{$\beta$}\rhd \tddeux{$\alpha$}{$\beta$}=2\ \tdtroisun{$\alpha$}{$\beta$}{$\beta$}+\tdtroisdeux{$\alpha$}{$\beta$}{$\beta$},
\end{align*}
which is different from $\tdun{$\beta$}\rhd_{\RT} \tddeux{$\alpha$}{$\beta$}$= \tdun{$\beta$}\tdun{$\alpha$}+\tdun{$\beta$}\tdun{$\beta$}
obtained in Example~\ref{exam:preL} by replacing $F_1$ by $\tdun{$\beta$}$.
\end{remark}

\smallskip

\noindent {\bf Acknowledgments}:
We thank Prof. Foissy for suggestions that have been very helpful to improve the quality of this paper.
Yi Zhang was supported by China Scholarship Council to visit Department of Mathematics at University of Southern California and he would like to thank Prof. Susan Montgomery for hospitality during his visit.
This work was supported by the National Natural Science Foundation
of China (Grant No.\@ 11771191 and 11501267).

\end{document}